\documentclass[11pt]{article}

\newtheorem{theorem}{Theorem}
\newtheorem{lemma}{Lemma}
\newtheorem{proof}{Proof}
\usepackage{hyperref}

\usepackage{amsmath}
\usepackage{graphicx}
\usepackage{amssymb}
\usepackage{graphicx}
\usepackage[dvipsnames,usenames]{color}
\usepackage{subfigure}
\input{epsf.tex}

\usepackage[affil-it]{authblk}
\title{On the power series approximations of a structured batch arrival two-class retrial system with weighted fair orbit queues}
\author{Ioannis Dimitriou\footnote{ idimit@math.upatras.gr}}
\affil{\small Department of Mathematics, 
University of Patras, P.O.~Box 
26500, Patras, Greece.}
\date{\small \today}
\begin{document}
\maketitle
\begin{abstract}
We provide power series approximations for a structured batch arrival single server retrial system with two infinite capacity \textit{weighted fair} orbit queues, i.e., the re-transmission rate of an orbit depends on the state of the other orbit queue. We consider both exponential and arbitrary distributed service times. In both cases we obtain power series expansions of the generating functions of the stationary joint orbit queue-length distributions, and provide a recursive approach to calculate their coefficients. We also show how to obtain the generating function of the stationary joint orbit queue-length distribution with the aid of a Riemann boundary value problem. Power series approximations are also provided for the model with two independent Poisson streams of jobs with single arrivals. Numerical illustrations are performed and show the accuracy of our approach. \\
\textbf{Keywords} Weighted-fair orbit queues, Power series approximation, Two-class retrial system, Probabilistic priority, Structured batch arrivals.
\end{abstract}

%\linenumbers
\section{Introduction}
In this paper we focus on the fundamental problem of investigating the delay performance in a two-class retrial system with constant retrial policy and adaptive re-transmission control. In particular, we consider a single server queueing system in which arriving jobs belong to $N=2$ different
types, say $P_{i}$, $P_{2}$. 

The arrival times of demands form a Poisson process with rate $\lambda$. We consider structured batch arrivals \cite{taka}, under which, at every arrival epoch, a group of jobs that may contain both types of jobs arrive.
%with two independent Poisson streams of arrivals and two \textit{coupled} orbit queues, initially investigated in \cite{dimpeis}. 
%Following \cite{dimpeis}, two types of jobs arrive according to a Poisson process with rate $\lambda_{k}$, $k=1,2$. 
The server can hold at most one job at a time, and if upon the arrival of a group of jobs, the server is busy, the $P_{k}$ jobs in the blocked group are routed to an infinite capacity orbit queue $k$, $k=1,2$. If an arriving group of jobs finds the server idle and contains both types of jobs, with probability $p_{k}$ a $P_{k}$ job, $k=1,2,$ occupies the server ($p_{1}+p_{2}=1$), while the rest jobs join the orbit according to their type; i.e., we employ a probabilistic priority setting.

Jobs from the orbits try to access the server according to a special state dependent constant retrial policy. In particular, if both orbits are non-empty, jobs from orbit $k$ try to access the server according to a constant retrial policy with rate $\phi_{k}\mu_{k}^{*}$, where $\phi_{1}=\xi$, $\phi_{2}=\bar{\xi}:=1-\xi$, where $0\leq\xi\leq1$. If orbit 1 (resp. orbit 2) is the only non-empty, it changes its re-transmission rate from $\xi\mu_{1}^{*}$ (resp. $\bar{\xi}\mu_{2}^{*}$) to $\mu_{1}^{*}$ (resp. $\mu_{2}^{*}$). In such a model an orbit queue is aware of the state of the other, i.e., we consider \textit{weighted-fair} orbit queues, and accordingly reconfigures its transmission parameters to improve its performance.
\paragraph{Applications}
Such a system can be found in the modelling of cooperative wireless network with adaptive control. Such a network operates as follows: There is a finite number of source users that transmit packets to a common destination node, and a finite number of relay nodes (i.e. orbit queues) that assist source users by retransmitting their blocked packets; e.g., \cite{dim2,mor,hunt,sad,pap1}. Moreover, due to the current trend towards dense networks and the spatial reuse of resources, we must take into account the interdependency among transmissions in the network planning. In such a case, the service rate of each relay node depends on the set of active relay nodes of the network (i.e., coupled relays). Such a characteristic appears also in cognitive radio \cite{sad}, and allows a node to exploit the ``idle" periods of the other node, achieving the full utilization of the shared channel. Moreover, it is clear that in communication networks \cite{sidi} a message is composed of packets (e.g., voice, data packets). We assume that the node assigns priority in a probabilistic manner.

Furthermore, our system can be also used to model bandwidth sharing of elastic flows in packet networks \cite{gui2,gui17}. weighted fair queueing protocol is certainly the most popular and the most commonly used in current packet
networks among all the
bandwidth sharing disciplines proposed so far in the technical literature.
\paragraph{Contribution and related work}
Our aim is to develop a technique based on the power series approximations of the joint probability generating functions (pgf) of the stationary orbit queue length distribution. We consider both exponentially and arbitrarily distributed service times. We distinguish the analysis since in the former case the model is described by a three-dimensional Markov process and can be seen as a Markov modulated random walk in the quarter plane. Its analysis leads to a matrix-form functional equation (i.e., a system of functional equations), and due to its \textit{special structure} we can reduce it into a scalar functional equation corresponding to one of the states of the (modulated) chain. The latter case leads to a scalar functional equation.% This scalar functional equation is treated using the theory of boundary value problems; see details in \cite{dimpeis,dim2}.

Moreover, it is the first time in the related literature that multi-class retrial systems with structured batch arrivals and an adaptive retransmission control is studied. Retrial systems with structured batch arrivals under the \textit{classical retrial policy} were studied in \cite{mou,mou1}.

The study of queueing systems using the theory of boundary value problems was developed in \cite{fay,fay1,bv}; see also \cite{avr,dimpeis,dim2,gui2,van} (not exhaustive list). For alternative approaches to analyze two-dimensional queueing models see \cite{ad} (compensation method), and \cite{blanc1,blanc2} (power series algorithm; PSA). Both Compensation method and PSA have been shown to be applied even in modulated random walks in quarter plane \cite{ad1,pss}. The PSA \cite{hou} is based on power series expansions of steady-state probabilities as functions of a certain parameter of the system, usually the load. Starting from the balance equations, the coefficients of the terms in the power series can be calculated iteratively. 

In our approach we generalize the method developed in \cite{walr} to develop power series approximations in a retrial system of \textit{weighted-fair} orbit queues with structured batch arrivals. In particular, for the exponentially distributed service times, we first construct a power series expansion in $\xi$ (rather than in load) for a bivariate pgf that corresponds to the set of states of an idle server, and then using the matrix functional equation we construct the power series expansion for the pgf that corresponds to the set of states of the busy server. In this direction, we develop an approach that can be applied to any related Markov-modulated random walk in the quarter plane. Moreover, it is the first time in the related literature that such an approach is applied to retrial systems. The coefficients of the power terms are iteratively calculated from a constant term, corresponding to $\xi=0$, and refers to the joint pgf of the priority retrial queue. Numerical results shown that PSA is very accurate for $\xi$ close to 0 (and by symmetry close to 1). Similar approach is applied to the case of arbitrarily distributed service times. We also provide expressions for the bivariate pgf in terms of a solution of a Riemann boundary value problem. The major advantage of the PSA approach over the method of the boundary value problems, is that the basic performance metrics are explicitly calculated by the input parameters without additional numerical effort. A detailed treatment for the model with two independent Poisson streams of jobs and single arrivals is also given.

The paper is summarized as follows. In Section \ref{main} we describe in detail the mathematical model. In sections \ref{gen}, and \ref{exp} we provide our main results and construct the power series expansions in $\xi$ for the unknown pgfs, for the case of the arbitrarily and exponentially distributed service times, respectively. Expressions for the pgf of the stationary joint orbit length distribution in terms of a solution of a Riemann boundary value problem are given in Section \ref{rbvp}. The case of the model with two independent Poisson streams and single arrivals, along with some spacial cases are treated in Section \ref{sa}. Approximations of the mean orbit lengths, using the PSA are given in Section \ref{perf}, while numerical validations are presented in Section \ref{num}. 
\section{The model}\label{main}
We consider a single server queue with batch arrivals, two infinite capacity orbit queues, and two types of
jobs, say $P_{k}$, $k=1,2$ that may arrive in the same batch. Let $M = (M_{1}, M_{2})$ be a random vector where $M_{k}$ is the number of $P_{k}$ jobs
in an arriving batch, and denote by $g_{m_{1},m_{2}}=P(M_{1}=m_{1},M_{2}=m_{2})$, where $m_{k}\geq0$ with $g_{0,0}=0$. Define further for $|z_{1}|\leq1$, $|z_{2}|\leq1$,
\begin{displaymath}
\begin{array}{c}
G(z_{1},z_{2})=\sum_{m_{1}=0}^{\infty}\sum_{m_{2}=0}^{\infty}g_{m_{1},m_{2}}z_{1}^{m_{1}}z_{2}^{m_{2}},\,\bar{g}_{k}=\mathbb{E}(M_{k}),\,k=1,2,
\end{array}
\end{displaymath}
and that the time between two succesive (batch) arrivals is exponentially
distributed with rate $\lambda$, independently of any other event in the system. The system dynamics are as follows:
\begin{enumerate}
\item If an arriving batch \textit{contains both types of jobs} and finds the server idle, a $P_{k}$ job in the batch occupies the server with probability $p_{k}$, $k=1,2$, $p_{1}+p_{2}=1$, and the remaining jobs join an infinite capacity orbit queue according to their type, i.e., we employ a probabilistic ``priority" for the primary arriving jobs. On the other, if the server is busy, all the $P_{k}$ jobs in batch, join orbit queue $k$, $k=1,2$.
\item If an arriving batch \textit{contains only one type of jobs}, say $P_{k}$, and finds the server idle, a $P_{k}$ job in the batch occupies the server and the remaining jobs join the orbit queue $k$, $k=1,2$. On the other, if the server is busy all the jobs in batch, join orbit queue $k$.
\item Jobs from the orbits try to access the server according to a state-dependent constant retrial policy. In particular, if
both orbit queues are non-empty, the orbit queue $k$ tries to re-transmit a blocked $P_{k}$ job to the server after an exponentially distributed time with
rate $\mu_{k}$. If an orbit queue is empty, then, the other orbit queue changes its re-transmission rate
from $\mu_{k}$ to $\mu_{k}^{*}$.
\item In this work we assume that $\mu_{k}=\phi_{k}\mu_{k}^{*}$, $k=1,2$, with $\phi_{1}=\xi$, $\phi_{2}=1-\xi$, where $0\leq\xi\leq1$, i.e., we consider ``weighted fair" orbit queues (WFOQ), or equivalently ``generalized processor sharing" orbit queues (GPSOQ).
\item Service times $S$ are independent and identically distributed (i.i.d.) random variables with p.d.f. $b(x)$, c.d.f. $B(x)$ and let $E(S)=\bar{b}$. 
\end{enumerate}
In the following we will distinguish the analysis for arbitrarily distributed and exponentially distributed service times. 
\section{The case of arbitrarily distributed service times}\label{gen}
Denote by $\tau_{m}$ the time of the $m$-th departure and $N_{k,m}$ the number of orbiting jobs in orbit queue $k$ just after the the time $\tau_{m}$. Define also by $A_{k,m}$, $k=1,2$ the number of type $k$ jobs that arrive during the $m$th service time, and let for $|z_{1}|\leq 1$, $|z_{2}|\leq 1$,
\begin{displaymath}
\begin{array}{rl}
\pi_{i,j}=\lim_{m\rightarrow \infty}Pr((N_{1,m},N_{2,m})=(i,j)),& \Pi(z_{1},z_{2})=\sum_{i=0}^{\infty}\sum_{j=0}^{\infty}\pi_{i,j}z_{1}^{i}z_{2}^{j},
\end{array}
\end{displaymath}
and
\begin{displaymath}
\begin{array}{rl}
P(A_{1}=k_{1},A_{2}=k_{2})=&d_{k_{1},k_{2}}=\int_{0}^{\infty}\sum_{n=1}^{\infty}e^{-\lambda t}\frac{(\lambda t)^{n}}{n!}g_{k_{1},k_{2}}^{(n)}dB(t),\\
d^{*}(z_{1},z_{2})=&\sum_{k_{1}=0}^{\infty}\sum_{k_{2}=0}^{\infty}d_{k_{1},k_{2}}z_{1}^{k_{1}}z_{2}^{k_{2}}=\beta^{*}(\lambda-\lambda G(z_{1},z_{2})),
\end{array}
\end{displaymath}
where $g_{k_{1},k_{2}}^{(n)}$ is the $n$-th convolution of the sequence $g_{k_{1},k_{2}}$. Then, it is readily follows that,
\begin{displaymath}
\begin{array}{rl}
\pi_{i,j}=&\frac{\mu_{1}^{*}\xi}{\lambda+\mu_{2}-\xi(\mu_{2}-\mu_{1}^{*})}\sum_{m_{1}=1}^{i+1}\sum_{m_{2}=1}^{j}\pi_{m_{1},m_{2}}d_{i+1-m_{1},j-m_{2}}\\&+\frac{\mu_{2}^{*}(1-\xi)}{\lambda+\mu_{2}-\xi(\mu_{2}-\mu_{1}^{*})}\sum_{m_{1}=1}^{i}\sum_{m_{2}=1}^{j+1}\pi_{m_{1},m_{2}}d_{i-m_{1},j+1-m_{2}}\\
&+\sum_{m_{1}=1}^{i}\sum_{m_{2}=1}^{j}\sum_{k_{1}=1}^{i-m_{1}+1}\sum_{k_{2}=1}^{j-m_{2}}\frac{\lambda g_{k_{1},k_{2}}p_{1}}{\lambda+\mu_{2}-\xi(\mu_{2}-\mu_{1}^{*})}\pi_{m_{1},m_{2}}d_{i-m_{1}-k_{1}+1,j-m_{2}-k_{2}}\\&+\sum_{m_{1}=1}^{i}\sum_{m_{2}=1}^{j}\sum_{k_{1}=1}^{i-m_{1}}\sum_{k_{2}=1}^{j-m_{2}+1}\frac{\lambda g_{k_{1},k_{2}}p_{2}}{\lambda+\mu_{2}-\xi(\mu_{2}-\mu_{1}^{*})}\pi_{m_{1},m_{2}}d_{i-m_{1}-k_{1},j-m_{2}-k_{2}+1}\\
&+\sum_{m_{1}=1}^{i}\sum_{m_{2}=1}^{j}\sum_{k_{1}=1}^{i-m_{1}+1}\frac{\lambda g_{k_{1},0}}{\lambda+\mu_{2}-\xi(\mu_{2}-\mu_{1}^{*})}\pi_{m_{1},m_{2}}d_{i-m_{1}-k_{1}+1,j-m_{2}-k_{2}}\\
&+\sum_{m_{1}=1}^{i}\sum_{m_{2}=1}^{j}\sum_{k_{2}=1}^{j-m_{2}+1}\frac{\lambda g_{0,k_{2}}}{\lambda+\mu_{2}-\xi(\mu_{2}-\mu_{1}^{*})}\pi_{m_{1},m_{2}}d_{i-m_{1},j-m_{2}-k_{2}+1}\\
&+\frac{\mu_{1}^{*}}{\lambda+\mu_{1}^{*}}\sum_{k_{1}=1}^{m+1}\pi_{k_{1},0}d_{m+1-k_{1},l}+\frac{\mu_{2}^{*}}{\lambda+\mu_{2}^{*}}\sum_{k_{2}=1}^{l+1}\pi_{0,k_{2}}d_{m,l+1-k_{2}}\\
&+\sum_{m_{1}=1}^{i}\sum_{k_{1}=1}^{i-m_{1}+1}\sum_{k_{2}=1}^{j}\frac{\lambda g_{k_{1},k_{2}}p_{1}}{\lambda+\mu_{1}^{*}}\pi_{m_{1},0}d_{i-m_{1}-k_{1}+1,j-k_{2}}\\
&+\sum_{m_{1}=1}^{i}\sum_{k_{1}=1}^{i-m_{1}+1}\frac{\lambda g_{k_{1},0}}{\lambda+\mu_{1}^{*}}\pi_{m_{1},0}d_{i-m_{1}-k_{1}+1,j}\\
&+\sum_{m_{1}=1}^{i}\sum_{k_{1}=1}^{i-m_{1}}\sum_{k_{2}=1}^{j+1}\frac{\lambda g_{k_{1},k_{2}}p_{2}}{\lambda+\mu_{1}^{*}}\pi_{m_{1},0}d_{i-m_{1}-k_{1},j-k_{2}+1}\\
&+\sum_{m_{1}=1}^{i}\sum_{k_{2}=1}^{j+1}\frac{\lambda g_{0,k_{2}}}{\lambda+\mu_{1}^{*}}\pi_{m_{1},0}d_{i-m_{1},j+1}\\

&+\sum_{m_{2}=1}^{j}\sum_{k_{1}=1}^{i-m_{1}+1}\sum_{k_{2}=1}^{j}\frac{\lambda g_{k_{1},k_{2}}p_{1}}{\lambda+\mu_{1}^{*}}\pi_{m_{1},0}d_{i-m_{1}-k_{1}+1,j-k_{2}}\\
&+\sum_{m_{1}=1}^{i}\sum_{k_{1}=1}^{i+1}\sum_{k_{2}=1}^{j}\frac{\lambda g_{k_{1},k_{2}}p_{1}}{\lambda+\mu_{2}^{*}}\pi_{0,m_{2}}d_{i-k_{1}+1,j-m_{2}-k_{2}}\\
&+\sum_{m_{2}=1}^{j}\sum_{k_{1}=1}^{i}\sum_{k_{2}=1}^{j-m_{2}+1}\frac{\lambda g_{k_{1},k_{2}}p_{2}}{\lambda+\mu_{2}^{*}}\pi_{0,m_{2}}d_{i-k_{1},j-m_{2}-k_{2}+1}\\
&+\sum_{m_{2}=1}^{j}\sum_{k_{1}=1}^{i+1}\frac{\lambda g_{k_{1},0}}{\lambda+\mu_{2}^{*}}\pi_{0,m_{2}}d_{i-m_{1}+1,j-m_{2}}\\
&+\sum_{m_{2}=1}^{j}\sum_{k_{2}=1}^{j-m_{2}+1}\frac{\lambda g_{0,k_{2}}}{\lambda+\mu_{2}^{*}}\pi_{0,m_{2}}d_{i,j-m_{2}-k_{2}+1}\\
&+\pi_{0,0}[\sum_{k_{1}=1}^{i+1}\sum_{k_{2}=1}^{j}\frac{\lambda g_{k_{1},k_{2}}p_{1}}{\lambda}d_{i-k_{1}+1,j-k_{2}}+\sum_{k_{1}=1}^{i+1}\frac{\lambda g_{k_{1},0}}{\lambda}d_{i-k_{1}+1,j}\\
&+\sum_{k_{1}=1}^{i}\sum_{k_{2}=1}^{j+1}\frac{\lambda g_{k_{1},k_{2}}p_{2}}{\lambda}d_{i-k_{1},j-k_{2}+1}+\sum_{k_{2}=1}^{j+1}\frac{\lambda g_{0,k_{2}}}{\lambda}d_{i,j-k_{2}+1}].
\end{array}
\end{displaymath}
By applying the generating function approach we come up with the following functional equation
\begin{equation}
\begin{array}{l}
(z_{1}z_{2}-\phi_{0}(z_{1},z_{2}))\Pi(z_{1},z_{2})=(z_{2}\phi_{1}(z_{1},z_{2})-\phi_{0}(z_{1},z_{2}))\Pi(z_{1},0)\\+(z_{1}\phi_{2}(z_{1},z_{2})-\phi_{0}(z_{1},z_{2}))\Pi(0,z_{2})\\+(\phi_{0}(z_{1},z_{2})-z_{1}\phi_{2}(z_{1},z_{2})-z_{2}\phi_{1}(z_{1},z_{2})+\frac{L(z_{1},z_{2})}{\lambda})\Pi(0,0),
\end{array}
\label{fun0}
\end{equation}
where,
\begin{equation}
\begin{array}{rl}
\phi_{0}(z_{1},z_{2})=&\frac{[\xi\mu_{1}^{*}z_{2}+\mu_{2}^{*}(1-\xi)z_{1}+L(z_{1},z_{2})]\beta^{*}(\lambda-\lambda G(z_{1},z_{2}))}{\lambda+\mu_{2}^{*}-\xi(\mu_{2}^{*}-\mu_{1}^{*})},\\
\phi_{1}(z_{1},z_{2})=&\frac{\mu_{1}^{*}+L(z_{1},z_{2})/z_{2}}{\lambda+\mu_{1}^{*}}\beta^{*}(\lambda-\lambda G(z_{1},z_{2})),\\
\phi_{2}(z_{1},z_{2})=&\frac{\mu_{2}^{*}+L(z_{1},z_{2})/z_{1}}{\lambda+\mu_{2}^{*}}\beta^{*}(\lambda-\lambda G(z_{1},z_{2})),\\
L(z_{1},z_{2})=&\lambda[(p_{1}z_{2}+p_{2}z_{1})G(z_{1},z_{2})\\&+(z_{2}-z_{1})(p_{2}G(z_{1},0)-p_{1}G(0,z_{2}))].
\end{array}
\end{equation}

We will first proceed with the investigation of the stability condition. Our model belongs to the general class of two dimensional random walks in quarter plane \cite{cohenRW}. Denote now by $\tau_{k}=\frac{d}{d_{z_{1}}}\phi_{k}(z_{1},z_{2})|_{z_{1}=z_{2}=1}$, $\nu_{k}=\frac{d}{d_{z_{2}}}\phi_{k}(z_{1},z_{2})|_{z_{1}=z_{2}=1}$, $k=0,1,2$. Following Th. 6.1, pp. 95-98, it is readily seen that in order the system to be stable, $\tau_{0}$, $\nu_{0}$ cannot be both greater than 1. For sake of clarity, assume hereon that $\tau_{0}<1$.
\begin{lemma}
If $\tau_{0}<1$, our system is stable if and only if
\begin{equation}
\begin{array}{l}
\rho:=\frac{\lambda}{\mu_{1}^{*}}[\bar{g}_{1}\bar{b}(\lambda+\mu_{1}^{*})+(\bar{g}_{1}-1+p_{2}(1-G(1,0))+p_{1}G(0,1))(1+\frac{\lambda\bar{g}_{2}\bar{b}(\mu_{1}^{*}-\mu_{2}^{*})}{\mu_{2}^{*}})]\\
+\frac{\lambda}{\mu_{2}^{*}}[\bar{g}_{2}\bar{b}(\lambda+\mu_{2}^{*})+(\bar{g}_{2}-1+p_{1}(1-G(0,1))+p_{2}G(1,0))(1+\frac{\lambda\bar{g}_{1}\bar{b}(\mu_{2}^{*}-\mu_{1}^{*})}{\mu_{1}^{*}})]<1
\end{array}
\end{equation}
\end{lemma}
\begin{proof}
Following Th. 6.1, \cite{cohenRW}, if $\tau_{0}<1$, our system is stable iff
\begin{equation}
\begin{array}{l}
\tau_{1}-1-\nu_{1}\frac{1-\tau_{1}}{1-\nu_{1}}=\tau_{2}-1-\nu_{2}\frac{1-\tau_{1}}{1-\nu_{1}}<0,
\end{array}
\label{sta}
\end{equation}
where
\begin{displaymath}
\begin{array}{rl}
\tau_{0}=&\frac{\lambda[\bar{g}_{1}\bar{b}(\lambda+\mu_{2}^{*})+\bar{g}_{1}+p_{2}(1-G(1,0))+p_{1}G(0,1)]+\mu_{2}^{*}-\xi[\mu_{2}^{*}+\lambda\bar{g}_{1}\bar{b}(\mu_{2}^{*}-\mu_{1}^{*})]}{\lambda+\mu_{2}^{*}-\xi(\mu_{2}^{*}-\mu_{1}^{*})},\\
\nu_{0}=&\frac{\lambda[\bar{g}_{2}\bar{b}(\lambda+\mu_{2}^{*})+\bar{g}_{2}+p_{1}(1-G(0,1))+p_{2}G(1,0)]+\xi[\mu_{1}^{*}+\lambda\bar{g}_{2}\bar{b}(\mu_{1}^{*}-\mu_{2}^{*})]}{\lambda+\mu_{2}^{*}-\xi(\mu_{2}^{*}-\mu_{1}^{*})},\\
\tau_{1}=&\frac{\lambda[\bar{g}_{1}\bar{b}(\lambda+\mu_{1}^{*})+\bar{g}_{1}+p_{2}(1-G(1,0))+p_{1}G(0,1)]}{\lambda+\mu_{1}^{*}},\\
\nu_{1}=&\frac{\lambda[\bar{g}_{2}\bar{b}(\lambda+\mu_{1}^{*})+\bar{g}_{1}-1+p_{1}(1-G(0,1))+p_{2}G(1,0)]}{\lambda+\mu_{1}^{*}},\\
\tau_{2}=&\frac{\lambda[\bar{g}_{1}\bar{b}(\lambda+\mu_{2}^{*})+\bar{g}_{1}-1+p_{2}(1-G(1,0))+p_{1}G(0,1)]}{\lambda+\mu_{2}^{*}},\\
\nu_{2}=&\frac{\lambda[\bar{g}_{2}\bar{b}(\lambda+\mu_{2}^{*})+\bar{g}_{1}+p_{1}(1-G(0,1))+p_{2}G(1,0)]}{\lambda+\mu_{2}^{*}},
\end{array}
\end{displaymath}
By substituting the $\tau_{k}$, $\nu_{k}$, $k=0,1,2,$ in (\ref{sta}) we obtain the desired result.
\end{proof}
\subsection{Main result}
In the following, we proceed with our main result. We first need to rewrite (\ref{fun0}) in the following form:
\begin{equation}
\begin{array}{l}
R(z_{1},z_{2})\Pi(z_{1},z_{2})=A(z_{1},z_{2})\Pi(z_{1},0)+B(z_{1},z_{2})\Pi(0,z_{2})+C(z_{1},z_{2})\Pi(0,0),
\end{array}
\label{fun}
\end{equation}
where
\begin{equation}
\begin{array}{rl}
R(z_{1},z_{2})=&z_{1}z_{2}(\lambda+\mu_{2}-\xi(\mu_{2}-\mu_{1}^{*}))\\&-[\xi\mu_{1}^{*}z_{2}+\mu_{2}^{*}(1-\xi)z_{1}+L(z_{1},z_{2})]\beta^{*}(\lambda-\lambda G(z_{1},z_{2})),
\end{array}
\end{equation}
\begin{displaymath}
\begin{array}{rl}
L(z_{1},z_{2})=&\lambda\{(p_{1}z_{2}+p_{2}z_{1})G(z_{1},z_{2})+(z_{2}-z_{1})(p_{2}G(z_{1},0)-p_{1}G(0,z_{2}))\},\vspace{2mm}\\
A(z_{1},z_{2})=&(1-\xi)\frac{\beta^{*}(\lambda-\lambda G(z_{1},z_{2}))}{\lambda+\mu_{1}^{*}}\\&\times[(\mu_{2}^{*}-\mu_{1}^{*})L(z_{1},z_{2})+\lambda(\mu_{1}^{*}z_{2}-\mu_{2}^{*}z_{1})+\mu_{1}^{*}\mu_{2}^{*}(z_{2}-z_{1})],\vspace{2mm}\\
B(z_{1},z_{2})=&-\frac{\xi}{1-\xi}\frac{\lambda+\mu_{1}^{*}}{\lambda+\mu_{2}^{*}}A(z_{1},z_{2}),\vspace{2mm}\\
C(z_{1},z_{2})=&\beta^{*}(\lambda-\lambda G(z_{1},z_{2}))[L(z_{1},z_{2})(1+(\lambda+\mu_{2}^{*}-\xi(\mu_{2}^{*}-\mu_{1}^{*}))\\&\times(\frac{1}{\lambda}-\frac{1}{\lambda+\mu_{1}^{*}}-\frac{1}{\lambda+\mu_{2}^{*}}))+\mu_{2}^{*}z_{1}-\xi(\mu_{2}^{*}z_{1}-\mu_{1}^{*}z_{2})\\&-(\lambda+\mu_{2}^{*}-\xi(\mu_{2}^{*}-\mu_{1}^{*}))(\frac{\mu_{1}^{*}z_{2}}{\lambda+\mu_{1}^{*}}+\frac{\mu_{2}^{*}z_{1}}{\lambda+\mu_{2}^{*}})].
\end{array}
\end{displaymath}
The functional equation (\ref{fun}) relates $\Pi(z_{1},z_{2})$ to $\Pi(z_{1},0)$, $\Pi(0,z_{2})$, $\Pi(0,0)$. We can solve it either using the theory of boundary value problems, or by applying the power series approximation. We will start with the power series approximation and then proceed with the solution in terms of a Riemann boundary value problem.

Assume hereon that $\rho<1$. We will construct a power series expansion of the pgf $\Pi(z_{1},z_{2})$ in $\xi$, directly from (\ref{fun}). Denote hereon $\Pi(z_{1},z_{2}):=\Pi(z_{1},z_{2},\xi)$ to express the dependence of pgf on $\xi$. We will show that 
\begin{equation}
\begin{array}{c}
\Pi(z_{1},z_{2};\xi)=\sum_{m=0}^{\infty}V_{m}(z_{1},z_{2})\xi^{m}.
\end{array}
\label{psl}
\end{equation}
and we will show how we can obtain recursively the terms $V_{m}(z_{1},z_{2})$.
 The functional equation (\ref{g1}) is rearranged as
\begin{equation}
\begin{array}{l}
U(z_{1},z_{2})\Pi(z_{1},z_{2},\xi)-S(z_{1},z_{2})\Pi(z_{1},0,\xi)-T_{0}(z_{1},z_{2})\Pi(0,0,\xi)\\
=\xi[G_{0}(z_{1},z_{2})\Pi(z_{1},z_{2},\xi)-S(z_{1},z_{2})(\Pi(z_{1},0,\xi)-\frac{\lambda+\mu_{1}^{*}}{\lambda+\mu_{2}^{*}}\Pi(0,z_{2},\xi))\\+T_{1}(z_{1},z_{2})\Pi(0,0,\xi)],
\end{array}
\label{r1l}
\end{equation}
where,
\begin{displaymath}
\begin{array}{rl}
U(z_{1},z_{2})=&z_{2}(\lambda+\mu_{2}^{*})-(L(z_{1},z_{2})z_{1}^{-1}+\mu_{2}^{*})\beta^{*}(\lambda-\lambda G(z_{1},z_{2})),\\
G_{0}(z_{1},z_{2})=&\mu_{2}^{*}(z_{2}-\beta^{*}(\lambda-\lambda G(z_{1},z_{2})))-\mu_{1}^{*}z_{2}(1-\frac{\beta^{*}(\lambda-\lambda G(z_{1},z_{2}))}{z_{1}}),\\
S(z_{1},z_{2})=&\frac{\beta^{*}(\lambda-\lambda G(z_{1},z_{2}))}{\lambda+\mu_{1}^{*}}[(\mu_{2}^{*}-\mu_{1}^{*})L(z_{1},z_{2})z_{1}^{-1}+\lambda(\mu_{1}^{*}z_{2}z_{1}^{-1}-\mu_{2}^{*})\\&+\mu_{1}^{*}\mu_{2}^{*}(z_{2}z_{1}^{-1}-1)],\\
T_{0}(z_{1},z_{2})=&\frac{[(\lambda+\mu_{2}^{*})\mu_{1}^{*}z_{1}^{-1}L(z_{1},z_{2}-\lambda z_{2})]\beta^{*}(\lambda-\lambda G(z_{1},z_{2}))}{(\lambda+\mu_{1}^{*})\lambda},\\
T_{1}(z_{1},z_{2})=&\frac{[(\mu_{2}^{*}-\mu_{1}^{*})\mu_{1}^{*}z_{1}^{-1}\lambda z_{2}-L(z_{1},z_{2})]\beta^{*}(\lambda-\lambda G(z_{1},z_{2}))}{(\lambda+\mu_{1}^{*})\lambda}+\frac{\lambda+\mu_{1}^{*}}{\lambda+\mu_{2}^{*}}S(z_{1},z_{2}).
\end{array}
\end{displaymath}

Under such a setting we succeed to isolate only one boundary function in the left hand side of (\ref{r1}). The following theorem summarizes our main result.
\begin{theorem}\label{th1}
Under stability condition (\ref{sta}), the pgf $\Pi(z_{1},z_{2},\xi)$ can be written in power series expansions on $\xi$ with coefficients
\begin{equation}
\begin{array}{rl}
V_{0}(z_{1},z_{2})=&\frac{(1-\rho)[T_{0}(z_{1},z_{2})S(z_{1},Y_{0}(z_{1}))-S(z_{1},z_{2})T_{0}(z_{1},Y_{0}(z_{1}))]}{(\bar{g}_{1}+\bar{g}_{2})U(z_{1},z_{2})S(z_{1},Y_{0}(z_{1}))},\\
V_{m}(z_{1},z_{2})=&\frac{Q_{m-1}(z_{1},z_{2})}{U(z_{1},z_{2})S(z_{1},Y_{0}(z_{1}))},\,m\geq 1,
\end{array}
\label{u0a}
\end{equation}
where for $m\geq0$,
\begin{displaymath}
\begin{array}{rl}
Q_{m}(z_{1},z_{2})=&G_{0}(z_{1},z_{2})S(z_{1},Y_{0}(z_{1}))V_{m}(z_{1},z_{2})\\&-G_{0}(z_{1},Y_{0}(z_{1}))S(z_{1},z_{2})V_{m}(z_{1},Y_{0}(z_{1}))\\&+\frac{(\lambda+\mu_{1}^{*})S(z_{1},z_{2})S(z_{1},Y_{0}(z_{1}))}{\lambda+\mu_{1}^{*}}[V_{m}(0,z_{2})-V_{m}(0,Y_{0}(z_{1}))]\\&
+V_{m}(0,0)[T_{1}(z_{1},z_{2})S(z_{1},Y_{0}(z_{1}))-S(z_{1},z_{2})T_{1}(z_{1},Y_{0}(z_{1}))],
\end{array}
\end{displaymath}
with $Q_{-1}(z_{1},z_{2})=0$.
\end{theorem}
\begin{proof}
Having in mind that $\Pi(z_{1},z_{2};\xi)$ is analytic function of $\xi$ in a neighborhood of 0 (see \ref{appe}, and \cite{walr}) we will express $\Pi(z_{1},z_{2})$ in power series expansion for all $z_{1}$, $z_{2}$ in the unit disk. Using (\ref{psl}), (\ref{r1l}) and equate the corresponding powers of $\xi$ at both sides yields
\begin{equation}
\begin{array}{rl}
U(z_{1},z_{2})V_{m}(z_{1},z_{2})=&S(z_{1},z_{2})V_{m}(z_{1},0)+P_{m-1}(z_{1},z_{2})\\&+T_{0}(z_{1},z_{2})V_{m}(0,0),\,m\geq0,
\end{array}\label{zma}
\end{equation} 
where 
\begin{displaymath}
\begin{array}{rl}
P_{m}(z_{1},z_{2})=&G_{0}(z_{1},z_{2})V_{m}(z_{1},z_{2})-S(z_{1},z_{2})[V_{m}(z_{1},0)-\frac{\lambda+\mu_{1}^{*}}{\lambda+\mu_{2}^{*}} V_{m}(0,z_{2})]\\&+T_{1}(z_{1},z_{2})V_{m}(0,0),\,m\geq1,
\end{array}
\end{displaymath}
with $P_{-1}(z_{1},z_{2})=0$. A simple application of Rouch\'{e}'s theorem ensures that for $|z_{1}|\leq1$, $U(z_{1},z_{2})=0$ has a unique root, say $Y_{0}(z_{1})$ such that $|Y_{0}(z_{1})|<1$, with $Y_{0}(1)=1$. Due to the implicit function theorem $Y_{0}(z_{1})$ is an analytic function in the unit disk, and and
\begin{displaymath}
\begin{array}{l}
\frac{d}{dz_{1}}Y_{0}(z_{1})|_{z_{1}=1}=\frac{\lambda[\bar{g}_{1}-1+\bar{g}_{1}(\lambda+\mu_{2}^{*})\bar{b}+p_{2}(1-G(1,0))+p_{1}G(0,1)]}{\mu_{2}^{*}-\lambda[\bar{g}_{2}-1+\bar{g}_{2}(\lambda+\mu_{2}^{*})\bar{b}+p_{1}(1-G(0,1))+p_{2}G(1,0)]}.
\end{array}
\end{displaymath}
Since $\Pi(z_{1},z_{2})$ is analytic in the unit disk, the coefficients $V_{m}(z_{1},z_{2})$ are also analytic, and thus, the right hand side of (\ref{zma}) vanishes for $z_{2}=Y_{0}(z_{1})$, and gives
\begin{equation}
\begin{array}{c}
V_{m}(z_{1},0)=-\frac{P_{m-1}(z_{1},Y_{0}(z_{1}))+T_{0}(z_{1},z_{2})V_{m}(0,0)}{S(z_{1},Y_{0}(z_{1})}.
\end{array}
\label{fga}
\end{equation}
Using (\ref{fga}), (\ref{zma}) we obtain for $m\geq0$
\begin{equation}
\begin{array}{rl}
V_{m}(z_{1},z_{2})=&\frac{1}{U(z_{1},z_{2})S(z_{1},Y_{0}(z_{1}))}[(T_{0}(z_{1},z_{2})S(z_{1},Y_{0}(z_{1}))\\&-S(z_{1},z_{2})T_{0}(z_{1},Y_{0}(z_{1})))V_{m}(0,0)+Q_{m-1}(z_{1},z_{2})],
\end{array}
\label{lpa}
\end{equation}
with $Q_{-1}(z_{1},z_{2})=0$. We now need to obtain $V_{0}(0,0)$. This constant will be found by $\Pi(1,1;\xi)=1$. After tedious algebra, This means that $V_{0}(1,1)=1$, $V_{m}^{(0)}(1,1)=0$, $m\geq1$, and $Q_{m}(1,1)=0$, $m\geq0$. Using (\ref{lpa}) for $m=0$, and setting $z_{1}=z_{2}=1$, we arrive after some algebra in $V_{0}(0,0)=\frac{1}{\bar{g}_{1}+\bar{g}_{2}}(1-\rho)$, $V_{m}^{(0)}(0,0)=0$, $m\geq1$. With this part and using (\ref{lpa}) we obtain (\ref{u0a}). 
\end{proof}
\section{The case of exponentially distributed service times}\label{exp}
Denote by $N_{k}(t)$ the number of jobs in orbit queue $k$, and by $C(t)$ the state of the server (i.e., busy or idle) at time $t$, respectively. Then $Q(t)=\left\{(N_{1}(t),N_{2}(t),C(t));t\geq0\right\}$ is an irreducible aperiodic Continuous Time Markov chain (CTMC) with state space $E=\left\{0,1,...\right\}\times\left\{0,1,...\right\}\times\left\{0,1\right\}$. Let us define,
\begin{displaymath}
\begin{array}{rl}
p_{i,j}(n)=&\lim_{t\to\infty}P(N_{1}(t)=i,N_{2}(t)=j,C(t)=n),\,(i,j,n)\in E,\\
H^{(n)}(z_{1},z_{2})=&\sum_{i=0}^{\infty}\sum_{j=0}^{\infty}p_{i,j}(n)z_{1}^{i}z_{2}^{j},\,n=0,1,\ \left|x\right|\leq 1,\ \left|y\right|\leq 1.
\end{array}
\end{displaymath}
Then, by writing down the balance equations we obtain the following system of functional equations
\begin{equation}
\begin{array}{rl}
\alpha H^{(0)}(z_{1},z_{2})-\mu H^{(1)}(z_{1},z_{2})=&(\mu_{2}^{*}-\mu_{1}^{*})[\bar{\xi}H^{(0)}(z_{1},0)-\xi H^{(0)}(0,z_{2})]\\
&+[\bar{\xi}\mu_{1}^{*}+\xi\mu_{2}^{*}]H^{(0)}(0,0),
\end{array}
\label{e111}
\end{equation}
\begin{equation}
\begin{array}{l}
(L(z_{1},z_{2})+\xi\mu_{1}^{*}z_{2}+\bar{\xi}\mu_{2}^{*}z_{1})H^{(0)}(z_{1},z_{2})-z_{1}z_{2}(\lambda-\lambda G(z_{1},z_{2})+\mu)H^{(1)}(z_{1},z_{2})\\
=(\mu_{2}^{*}z_{1}-\mu_{1}^{*}z_{2})[\bar{\xi}H^{(0)}(z_{1},0)-\xi H^{(0)}(0,z_{2})]+[\bar{\xi}\mu_{1}^{*}z_{2}+\xi\mu_{2}^{*}z_{1}]H^{(0)}(0,0).
\end{array}
\label{e211}
\end{equation} 
where $\alpha=\lambda+\xi\mu_{1}^{*}+\bar{\xi}\mu_{2}^{*}$, $\lambda=\lambda_{1}+\lambda_{2}$. Using (\ref{e111}), (\ref{e211}) we come up with the fundamental functional equation
\begin{equation}
\begin{array}{rl}
K(x,y)H^{(0)}(z_{1},z_{2})=&K_{1}(z_{1},z_{2})H^{(0)}(z_{1},0)+K_{2}(z_{1},z_{2})H^{(0)}(0,z_{2})\\&+K_{3}(z_{1},z_{2})H^{(0)}(0,0),
\end{array}
\label{g1}
\end{equation}
where for $\widehat{\lambda}_{k}=\lambda_{k}\alpha$, $k=1,2,$
\begin{equation}
\begin{array}{rl}
K(z_{1},z_{2})=&\mu[L(z_{1},z_{2})-\lambda z_{1}z_{2}+\xi\mu_{1}^{*}z_{2}(1-z_{1})+\bar{\xi}\mu_{2}^{*}z_{1}(1-z_{2})]\\&-z_{1}z_{2}\widehat{\lambda}(1-G(z_{1},z_{2})),\\
K_{1}(z_{1},z_{2})=&\bar{\xi}\{\mu[\mu_{2}^{*}z_{1}(1-z_{2})-\mu_{1}^{*}z_{2}(1-z_{1})\\&-\lambda z_{1}z_{2}(1-G(z_{1},z_{2}))(\mu_{2}^{*}-\mu_{1}^{*})]\},\\
K_{2}(z_{1},z_{2})=&-\frac{\xi K_{1}(z_{1},z_{2})}{1-\xi},\\
K_{3}(z_{1},z_{2})=&\mu[\bar{\xi}\mu_{1}^{*}z_{2}(1-z_{1})+\xi\mu_{2}^{*}z_{1}(1-z_{2})]\\&-\lambda z_{1}z_{2}(1-G(z_{1},z_{2}))(\bar{\xi}\mu_{1}^{*}+\xi\mu_{2}^{*}).
\end{array}
\label{ui1}
\end{equation}
Similarly, equation (\ref{g1}) relates $H^{(0)}(z_{1},z_{2})$ with $H^{(0)}(z_{1},0)$, $H^{(0)}(0,z_{2})$, $H^{(0)}(0,0)$. Assume hereon that $\rho<1$. Before proceeding further, we are firstly going to obtain the probability of a busy (idle) server using (\ref{e111}) and some side conditions. 

For each $k\leq i$, $i = 0, 1,...$, we consider the vertical cut between the states
$\{N_{1} = k, C = 1\}$ and $\{N_{1} = i + 1, C = 0\}$. Then,
\begin{displaymath}
\begin{array}{rl}
\lambda\sum_{k=0}^{i}p_{k,.}(1)g_{i-k+1,.}=&\mu_{1}^{*}p_{i+1,0}(0)+\xi\mu_{1}^{*}\sum_{m=1}^{\infty}p_{i+1,m}(0)\\
=&\xi\mu_{1}^{*}p_{i+1,.}(0)+\bar{\xi}\mu_{1}^{*}p_{i+1,0}(0).
\end{array}
\end{displaymath}
Summing for all $i$, we obtain
\begin{equation}
\begin{array}{rl}
\lambda(1-G(0,1))H^{(1)}(1,1)=&\xi\mu_{1}^{*}(H^{(0)}(1,1)-H^{(0)}(0,1))-\bar{\xi}\mu_{1}^{*}(H^{(0)}(1,0)-H^{(0)}(0,0)).
\end{array}
\label{hj}
\end{equation}
Similarly,
\begin{equation}
\begin{array}{rl}
\lambda(1-G(1,0))H^{(1)}(1,1)=&\bar{\xi}\mu_{2}^{*}(H^{(0)}(1,1)-H^{(0)}(1,0))-\xi\mu_{2}^{*}(H^{(0)}(0,1)-H^{(0)}(0,0)).
\end{array}
\label{hj0}
\end{equation}
Summing (\ref{hj}), (\ref{hj0}), using (\ref{e111}), and the fact that $H^{(1)}(1,1)+H^{(0)}(1,1)=1$, we obtain
\begin{equation}
\begin{array}{lr}
H^{(1)}(1,1)=\frac{\lambda}{\mu-\lambda(1-G(0,1)-G(1,0))},&H^{(0)}(1,1)=\frac{\mu-\lambda(2-G(0,1)-G(1,0))}{\mu-\lambda(1-G(0,1)-G(1,0))}.
\end{array}\label{laap}
\end{equation}

Our next steps are as follows: First, we are going to construct a power series expansion of the pgf $H^{(0)}(z_{1},z_{2})$ in $\xi$, directly from (\ref{g1}). Second, using (\ref{e111}), we construct the corresponding power series expansion for $H^{(1)}(z_{1},z_{2})$ in terms of those of $H^{(0)}(z_{1},z_{2})$. Denote hereon $H^{(n)}(z_{1},z_{2};\xi):=H^{(0)}(z_{1},z_{2})$ to express the dependence of pgfs on $\xi$. We will now show that 
\begin{equation}
\begin{array}{c}
H^{(n)}(z_{1},z_{2};\xi)=\sum_{m=0}^{\infty}V_{m}^{(n)}(z_{1},z_{2})\xi^{m},\,n=0,1,
\end{array}
\label{ps}
\end{equation}
and we will show how we can obtain recursively the terms $V_{m}^{(n)}(x,y)$.
 The functional equation (\ref{g1}) is rearranged as
\begin{equation}
\begin{array}{l}
U_{1}(z_{1},z_{2})H^{(0)}(z_{1},z_{2};\xi)-U_{0}(z_{1},z_{2})H^{(0)}(z_{1},0;\xi)-F_{0}(z_{1},z_{2})H^{(0)}(0,0;\xi)\\
=\xi U_{0}(z_{1},z_{2})[H^{(0)}(z_{1},z_{2};\xi)-H^{(0)}(z_{1},0;\xi)-H^{(0)}(0,z_{2};\xi)+H^{(0)}(0,0;\xi)],
\end{array}
\label{r1}
\end{equation}
where, 
\begin{displaymath}
\begin{array}{rl}
U_{1}(z_{1},z_{2})=&\mu[L(z_{1},z_{2})z_{1}^{-1}-\lambda z_{2}-\mu_{2}^{*}(1-z_{2})]-z_{2}\lambda(\lambda+\mu_{2}^{*})(1-G(z_{1},z_{2})),\\
U_{0}(z_{1},z_{2})=&\mu[\mu_{2}^{*}(1-z_{2})-\mu_{1}^{*}z_{2}(z_{1}^{-1}-1)]-z_{2}\lambda(\mu_{2}^{*}-\mu_{1}^{*})(1-G(z_{1},z_{2})),\\
F_{0}(z_{1},z_{2})=&\mu_{1}^{*}z_{2}[\mu(z_{1}^{-1}-1)-\lambda(1-G(z_{1}z_{2}))].
\end{array}
\end{displaymath}
Note that (\ref{r1}) has the same structure as (\ref{r1l}). Thus, in order to construct power series expansions of $H^{(0)}(z_{1},z_{2})$, we use (\ref{lp}) and proceed analogously as we did in Theorem \ref{th1}. Then, using (\ref{e111}), we can obtain the power series expansions of $H^{(1)}(z_{1},z_{2})$. The following theorem summarizes our main result.
\begin{theorem}\label{tho}
Under stability condition (\ref{sta}) the pgfs $H^{(n)}(x,y;\xi)$ can be written in power series expansions on $\xi$ with coefficients
\begin{equation}
\begin{array}{rl}
V_{0}^{(0)}(z_{1},z_{2})=&\frac{(1-\rho)[F_{0}(z_{1},z_{2})U_{0}(z_{1},\tilde{Y}_{0}(z_{1}))-U_{0}(z_{1},z_{2})F_{0}(z_{1},\tilde{Y}_{0}(z_{1}))]}{(\bar{g}_{1}+\bar{g}_{2})U_{1}(z_{1},z_{2})U_{0}(z_{1},\tilde{Y}_{0}(z_{1}))},\\V_{m}^{(0)}(z_{1},z_{2})=&\frac{U_{0}(z_{1},z_{2})\tilde{Q}_{m-1}(z_{1},z_{2})}{U_{1}(z_{1},z_{2})},\,m\geq 1,
\end{array}\label{u0}
\end{equation}
\begin{equation}
\begin{array}{rl}
V_{0}^{(1)}(z_{1},z_{2})=&\frac{\lambda+\mu_{2}^{*}}{\mu}V_{0}^{(0)}(z_{1},z_{2})-\frac{\mu_{2}^{*}-\mu_{1}^{*}}{\mu}V_{0}^{(0)}(z_{1},0)-\frac{(1-\rho)\mu_{1}^{*}}{(\bar{g}_{1}+\bar{g}_{2})\mu},\\
V_{1}^{(1)}(z_{1},z_{2})=&\frac{\lambda+\mu_{2}^{*}}{\mu}V_{1}^{(0)}(z_{1},z_{2})-\frac{\mu_{2}^{*}-\mu_{1}^{*}}{\mu}[V_{0}^{(0)}(z_{1},z_{2})-V_{0}^{(0)}(z_{1},0)\\&-V_{0}^{(0)}(0,z_{2})+V_{1}^{(0)}(z_{1},0)+\frac{1-\rho}{(\bar{g}_{1}+\bar{g}_{2})}],\\
V_{m}^{(1)}(z_{1},z_{2})=&\frac{\lambda+\mu_{2}^{*}}{\mu}V_{m}^{(0)}(z_{1},z_{2})\\&-\frac{\mu_{2}^{*}-\mu_{1}^{*}}{\mu}[V_{m-1}^{(0)}(z_{1},z_{2})-V_{m-1}^{(0)}(z_{1},0)+V_{m}^{(0)}(z_{1},0)],\,m\geq2,
\end{array}
\label{u1}
\end{equation}
where $\tilde{Y}_{0}(z_{1})$, $|z_{1}|\leq1$ is the only zero of $U_{1}(z_{1},z_{2})$ inside the unit disk $|z_{2}|\leq1$ and $\tilde{Q}_{m}(z_{1},z_{2})=V_{m}^{(0)}(z_{1},z_{2})-V_{m}^{(0)}(z_{1},\tilde{Y}_{0}(z_{1}))-V_{m}^{(0)}(0,z_{2})+V_{m}^{(0)}(0,\tilde{Y}_{0}(z_{1}))$, $m\geq0$, with $\tilde{Q}_{-1}(z_{1},z_{2})=0$.
\end{theorem}
\begin{proof}
The proof follows the lines of Theorem \ref{th1}. Note also that now $H^{(0)}(z_{1},z_{2};\xi)$ is analytic function of $\xi$ in a neighborhood of 0 (see \ref{appe}). We firstly express $H^{(0)}(z_{1},z_{2})$ in power series expansion using (\ref{ps}), (\ref{r1}) and equate the corresponding powers of $\xi$ at both sides. Then, by showing using Rouche's theorem that $U_{1}(z_{1},z_{2})=0$ has a unique root, say $Y_{0}(x)$ such that $|Y_{0}(x)|<1$., we obtain after some algebra
\begin{equation}
\begin{array}{c}
V_{m}^{(0)}(z_{1},0)=-\frac{F_{0}(z_{1},Y_{0}(z_{1}))}{U_{0}(z_{1},Y_{0}(z_{2}))}V_{m}^{(0)}(0,0)-P_{m-1}(z_{1},Y_{0}(z_{2})).
\end{array}
\label{fg}
\end{equation}
where $P_{m}(z_{1},z_{2})=V_{m}^{(0)}(z_{1},z_{2})-V_{m}^{(0)}(z_{1},0)-V_{m}^{(0)}(0,z_{2})+V_{m}^{(0)}(0,0)$, $m\geq1$, and $P_{-1}(z_{1},z_{2})=0$. 

Using (\ref{fg}), we obtain for $m\geq 0$,
\begin{equation}
\begin{array}{rl}
V_{m}^{(0)}(z_{1},z_{2})=&\frac{1}{U_{1}(z_{1},z_{2})}[\frac{F_{0}(z_{1},z_{2})U_{0}(z_{1},\tilde{Y}_{0}(z_{1}))-U_{0}(z_{1},z_{2})F_{0}(z_{1},\tilde{Y}_{0}(z_{1}))}{U_{0}(z_{1},\tilde{Y}_{0}(z_{1}))}V_{m}^{(0)}(0,0)\\&+U_{0}(z_{1},z_{2})Q_{m-1}(z_{1},z_{2})].
\end{array}
\label{lap}
\end{equation}

We now need to obtain $V_{0}^{(0)}(0,0)$. This constant will be found by using the fact that $H^{(0)}(1,1;\xi)=\frac{\mu-\lambda(2-G(0,1)-G(1,0))}{\mu-\lambda(1-G(0,1)-G(1,0))}$. This means that $V_{0}^{(0)}(1,1)=\frac{\mu-\lambda(2-G(0,1)-G(1,0))}{\mu-\lambda(1-G(0,1)-G(1,0))}$, $V_{m}^{(0)}(1,1)=0$, $m\geq1$, and $\tilde{Q}_{m}(1,1)=0$, $m\geq0$. Using (\ref{lap}) for $m=0$, and setting $z_{1}=z_{2}=1$, we arrive after some algebra in $V_{0}^{(0)}(0,0)=1-\rho$, $V_{m}^{(0)}(0,0)=0$, $m\geq1$. With this part and using (\ref{lap}) we obtain (\ref{u0}). Now substitute (\ref{ps}) in (\ref{e111}) to obtain,
\begin{displaymath}
\begin{array}{l}
\sum_{m=0}^{\infty}V_{m}^{(1)}(z_{1},z_{2})\xi^{m}=\frac{\alpha}{\mu}\sum_{m=0}^{\infty}V_{m}^{(0)}(z_{1},z_{2})\xi^{m}-\frac{\mu_{2}^{*}-\mu_{1}^{*}}{\mu}\sum_{m=0}^{\infty}V_{m}^{(0)}(z_{1},0)\xi^{m}\\
+\frac{\mu_{2}^{*}-\mu_{1}^{*}}{\mu}\sum_{m=0}^{\infty}[V_{m}^{(0)}(z_{1},0)+V_{m}^{(0)}(0,z_{2})]\xi^{m+1}-\frac{(1-\rho)\mu_{1}^{*}}{(\bar{g}_{1}+\bar{g}_{2})\mu}-\frac{(\mu_{2}^{*}-\mu_{1}^{*})(1-\rho)\xi}{(\bar{g}_{1}+\bar{g}_{2})\mu}.
\end{array}
\end{displaymath}
Equate the coefficients of the corresponding powers in $\xi$ to obtain $V_{m}^{(1)}(z_{1},z_{2})$ in terms of $V_{m}^{(0)}(z_{1},z_{2})$, as given in (\ref{u1}).
\end{proof}

Hence, starting from $V_{0}^{(0)}(z_{1},z_{2})$ in (\ref{u0}), we can iteratively determined all the functions $V_{m}^{(0)}(z_{1},z_{2})$, and subsequently using (\ref{u1}), the rest functions $V_{m}^{(1)}(z_{1},z_{2})$.
\section{Reduction to a Riemann boundary value problem}\label{rbvp}
We will only focus on the case of arbitrarily distributed service times. Note that (\ref{fun0}) is rewritten as 
\begin{equation}
\begin{array}{l}
(z_{1}z_{2}-\phi_{0}(z_{1},z_{2}))\Pi(z_{1},z_{2})\\=T(z_{1},z_{2})[\bar{\xi}(\lambda+\mu_{2}^{*})\Pi(z_{1},0)-\xi(\lambda+\mu_{1}^{*})\Pi(0,z_{2})]+W(z_{1},z_{2})\Pi(0,0),
\end{array}
\label{fun00}
\end{equation}
where
\begin{displaymath}
\begin{array}{rl}
T(z_{1},z_{2})=&\frac{L(z_{1},z_{2})(\mu_{2}^{*}-\mu_{1}^{*})+\mu_{1}^{*}z_{2}(\lambda+\mu_{2}^{*})-\mu_{2}^{*}z_{1}(\lambda+\mu_{1}^{*})}{(\lambda+\xi\mu_{1}^{*}+(1-\xi)\mu_{2}^{*})(\lambda+\mu_{2}^{*})(\lambda+\mu_{1}^{*})}\beta^{*}(\lambda-\lambda G(z_{1},z_{2})),\\
W(z_{1},z_{2})=&\frac{\beta^{*}(\lambda-\lambda G(z_{1},z_{2}))\mu_{2}^{*}(L(z_{1},z_{2})-\lambda z_{1})}{\lambda(\lambda+\mu_{2}^{*})}-\bar{\xi}(\lambda+\mu_{2}^{*})T(z_{1},z_{2}).
\end{array}
\end{displaymath}
The analysis of the \textit{kernel}, $z_{1}z_{2}-\phi_{0}(z_{1},z_{2})$ is the starting point for the determination of $\Pi(z_{1},z_{2})$, which is regular for $|z_{1}|<1$, continuous for $|z_{1}|\leq 1$ for every fixed $z_{2}$ with $|z_{2}|\leq 1$; and similarly, with $z_{1}$, $z_{2}$ interchanged.

Such types of random walks in the quarter plane, in which the steps to the West, South-West
and South are at most one, are discussed in \cite{coh,cohenRW,bv}. We will briefly show how we can determine $\Pi(z_{1},z_{2})$ via a transformation to a two dimensional
boundary value problem of mathematical physics, like a Riemann or
Riemann-Hilbert problem. It is easily noted that $\phi_{0}(0,0)=0$, which means that we do not allow one-transitions to south-west. The treatment of such kind of kernel is discussed in detail in \cite{bv}, part II, Sec. 3,10-3.12.
\subsection{Analysis of the kernel}
The zeros of the kernel $z_{1}z_{2}-\phi_{0}(z_{1},z_{2})$ in the region $|z_{1}|<1$, $|z_{2}|<1$, where the transform $\Pi(z_{1},z_{2})$ is finite, yields
\begin{equation}
T(z_{1},z_{2})[\bar{\xi}(\lambda+\mu_{2}^{*})\Pi(z_{1},0)-\xi(\lambda+\mu_{1}^{*})\Pi(0,z_{2})]+W(z_{1},z_{2})\Pi(0,0)=0.
\end{equation}
We follow \cite{bv}, and consider the kernel for 
\begin{displaymath}
z_{1}=gs,\ z_{2}=gs^{-1},\,|s|=1,\,|g|\leq1.
\end{displaymath}
Consequently,
\begin{equation}
\begin{array}{rl}
g^{2}=\phi_{0}(gs,gs^{-1}).
\end{array}\label{mc}
\end{equation}
The existence of the zeros of (\ref{mc}), i.e., of the kernel of (\ref{fun00}), in the region $|s|=1$, $|g|\leq1$ can be shown with Rouch\'{e}'s theorem. In particular, since $\phi_{0}(0,0)=0$, one of these zeros is zero, and the other one say $g=g(s)$. Define,
\begin{displaymath}
\begin{array}{rl}
S_{1}:=\left\{z_{1}:\ z_{1}=g(s)s,\ |s|=1\right\},& S_{2}:=\left\{z_{2}:\ z_{2}=g(s)s^{-1},\ |s|=1\right\}.
\end{array}
\end{displaymath}A point of major concern is to show that the contours $S_1$, $S_2$ are not selfintersecting. Moreover, further investigation is needed in order to clarify the position of $0$ with respect to $S_{1}$, $S_{2}$. Therefore, we will assume hereon that $S_{1}$, $S_{2}$ are Jordan contours . It turns out that $S_{1}$, $S_{2}$ traverse twice if $s$ traverses $|s|=1$ once. Moreover, following \cite{bv}, Lemma 10.2
\begin{enumerate}
\item If $\xi\mu_{1}^{*}>\bar{\xi}\mu_{2}^{*}$, then $z_{1}=0\in S_{1}^{-}$ and $z_{2}=0\in S_{2}^{+}$,
\item If $\xi\mu_{1}^{*}=\bar{\xi}\mu_{2}^{*}$, then $z_{1}=0\in S_{1}$ and $z_{2}=0\in S_{2}$,
\item If $\xi\mu_{1}^{*}<\bar{\xi}\mu_{2}^{*}$, then $z_{1}=0\in S_{1}^{+}$ and $z_{2}=0\in S_{2}^{-}$,
\end{enumerate}
where $S_{i}^{+}$, $S_{i}^{-}$ denote the interior and exterior of the contour $S_{i}$, $i=1,2$, respectively. Without loss of generality we will assume hereon that  $\xi\mu_{1}^{*}>\bar{\xi}\mu_{2}^{*}$.

Following \cite{bv}, Section II 3.6, we have to firstly consider the following boundary value problem for the functions $z_{1}:=g(s)s$, $z_{2}:=g(s)s^{-1}$. In particular, one can show that for the curves $S_{1}$, $S_{2}$, we can construct in the $x-$plane a smooth contour $L$ and a pair of mappings $z_{1}(x),\ x\in L^{+}\cup L$, $z_{2}(x),\ x\in L^{-}\cup L$ such that:
\begin{enumerate}
\item $z_{1}(x)$ is regular and univalent for $x\in L^{+}$, continuous for $x\in L^{+}\cup L$;\\ $z_{2}(x)$ is regular and univalent for $x\in L^{-}$, continuous for $x\in L^{-}\cup L$.
\item $z_{1}(x)$ maps $L^{+}$ conformally onto $S_{1}^{+}$;\\
$z_{2}(x)$ maps $L^{-}$ conformally onto $S_{2}^{+}$.
\item $z_{1}^{+}(x)$, $z_{2}^{-}(x)$, $x\in L$ is a zero pair of the kernel equation $z_{1}z_{2}=\phi_{0}(z_{1},z_{2})$, where
\begin{displaymath}
z_{1}^{+}(x)=\lim_{y\to x,\ y\in L^{+}}z_{1}(y),\,z_{2}^{-}(x)=\lim_{y\to x,\ y\in L^{-}}z_{2}(y).
\end{displaymath}
\item $z_{1}(0)>0$, $z_{1}(1)=1$, $z_{2}(\infty)=0$, $0<d:=\lim_{|x|\to \infty}|xz_{2}(x)|<\infty$.
\end{enumerate}

We proceed with the determination of $L$, and the mappings $z_{1}(.)$, $z_{2}(.)$. Since for $x\in L$, $(z_{1}^{+}(x),z_{2}^{-}(x))$ is a zero pair of the kernel equation $z_{1}z_{2}=\phi_{0}(z_{1},z_{2})$ for $|z_{1}|\leq 1$, $|z_{2}|\leq 1$ with $z_{1}^{+}(x)\in S_{1}$, $z_{2}^{-}(x)\in S_{2}$, we may write 
\begin{equation}
\begin{array}{c}
z_{1}^{+}(x)=g(e^{i\lambda(x)})e^{i\lambda(x)},\,z_{2}^{-}(x)=g(e^{i\lambda(x)})e^{-i\lambda(x)},
\end{array}
\label{uiop}
\end{equation}
where $\lambda(.):L\to[0,\pi]$, $\lambda(1)=0$. Then it is seen that for $x\in L$:
\begin{equation}
\begin{array}{l}
\log z_{1}^{+}(x)+\log\frac{xz_{2}^{-}(x)}{d}=\log\frac{g^{2}(e^{i\lambda(x)})}{d}+\log x,\\
\log z_{1}^{+}(x)-\log\frac{xz_{2}^{-}(x)}{d}=2i\lambda(x)-\log x+\log d.
\end{array}
\label{d}
\end{equation}
The solution of the above boundary value problem is:
\begin{displaymath}
\begin{array}{rl}
z_{1}(x)=&\exp\left\{\frac{1}{2\pi i}\int_{\zeta\in L}[\log\left\{g(e^{i\lambda(\zeta)})\zeta^{1/2}\right\}]\left\{\frac{\zeta+x}{\zeta-x}-\frac{\zeta+1}{\zeta-1}\frac{d\zeta}{\zeta}\right\}\right\},\,x\in L^{+},\\
z_{2}(x)=&x^{-1}\exp\left\{-\frac{1}{2\pi i}\int_{\zeta\in L}[\log\left\{g(e^{i\lambda(\zeta)})\zeta^{1/2}\right\}]\left\{\frac{\zeta+x}{\zeta-x}-\frac{\zeta+1}{\zeta-1}\frac{d\zeta}{\zeta}\right\}\right\},\,x\in L^{-}.
\end{array}
\end{displaymath}
Applying Plemelji-Sokhotski formulas for $x\in L$ gives:
\begin{displaymath}
\begin{array}{rl}
z_{1}^{+}(x)=&g(e^{i\lambda(x)})x^{1/2}\exp\left\{\frac{1}{2\pi i}\int_{\zeta\in L}[\log\left\{g(e^{i\lambda(\zeta)})\zeta^{1/2}\right\}]\left\{\frac{\zeta+x}{\zeta-x}-\frac{\zeta+1}{\zeta-1}\frac{d\zeta}{\zeta}\right\}\right\},\\
z_{2}^{-}(x)=&g(e^{i\lambda(x)})x^{-1/2}\exp\left\{-\frac{1}{2\pi i}\int_{\zeta\in L}[\log\left\{g(e^{i\lambda(\zeta)})\zeta^{1/2}\right\}]\left\{\frac{\zeta+x}{\zeta-x}-\frac{\zeta+1}{\zeta-1}\frac{d\zeta}{\zeta}\right\}\right\}.
\end{array}
\end{displaymath}
Substituting back in (\ref{d}) gives the following relation for the determination of $L$, $\lambda(x)$, $x\in L$:
\begin{displaymath}
\begin{array}{c}
\exp\left\{i\lambda(x)\right\}=x^{1/2}\exp\left\{\frac{1}{2\pi i}\int_{\zeta\in L}[\log\left\{g(e^{i\lambda(\zeta)}\zeta^{1/2}\right\}]\left\{\frac{\zeta+x}{\zeta-x}-\frac{\zeta+1}{\zeta-1}\frac{d\zeta}{\zeta}\right\}\right\}.
\end{array}
\end{displaymath}

Since for $x\in L$, $(z_{1}^{+}(x),z_{2}^{-}(x))$ is a zero pair of the kernel, it readily follows from (\ref{fun00}) that
\begin{displaymath}
\begin{array}{r}
T(z_{1}^{+}(x),z_{2}^{-}(x))[\bar{\xi}(\lambda+\mu_{2}^{*})\Pi(z_{1}^{+}(x),0)-\xi(\lambda+\mu_{1}^{*})\Pi(0,z_{2}^{-}(x))]\vspace{2mm}\\
=(\lambda+\mu_{2}^{*})(\lambda+\mu_{1}^{*})W(z_{1}^{+}(x),z_{2}^{-}(x))\Pi(0,0).
\end{array}
\end{displaymath}
Using (\ref{uiop}) we conclude in:
\begin{equation}
\Pi(z_{1}^{+}(x),0)=\frac{\xi(\lambda+\mu_{1}^{*})}{\bar{\xi}(\lambda+\mu_{2}^{*})}\Pi(0,z_{2}^{-}(x))+J(x),
\label{rbv}
\end{equation}
where
\begin{displaymath}
J(x)=\frac{(\lambda+\mu_{1}^{*})W(g(e^{i\lambda(x)})e^{i\lambda(x)},g(e^{i\lambda(x)})e^{-i\lambda(x)})}{\bar{\xi}T(g(e^{i\lambda(x)})e^{i\lambda(x)},g(e^{i\lambda(x)})e^{-i\lambda(x)})}.
\end{displaymath}
Note that (\ref{rbv}) defines a Riemann boundary value problem: For a simply connected Jordan contour $L$ and the mappings $z_{1}(x),\ x\in L\cup L^{+}$, $z_{2}(x),\ x\in L\cup L^{-}$, find two functions such that
\begin{enumerate}
\item $\Pi(z_{1}(x),0)$ is regular for $x\in L^{+}$, continuous  for $x\in L\cup L^{+}$;\\
$\Pi(0,z_{2}(x))$ is regular for $x\in L^{-}$, continuous  for $x\in L\cup L^{-}$,
\item satisfying the boundary condition (\ref{rbv}).
\end{enumerate}

Since the $\frac{\xi(\lambda+\mu_{1}^{*})}{\bar{\xi}(\lambda+\mu_{2}^{*})}$ never vanishes, the index of out boundary value problem is zero. Moreover, $\frac{\xi(\lambda+\mu_{1}^{*})}{\bar{\xi}(\lambda+\mu_{2}^{*})}$ satisfies the Holder condition on $L$. The solution of this Riemann boundary value problem is given by:
\begin{equation}
\begin{array}{rl}
\bar{\xi}(\lambda+\mu_{2}^{*})\Pi(z_{1}(x),0)=&\frac{1}{2i\pi}\int_{\zeta\in L}J(\zeta)\frac{d\zeta}{\zeta-x}+\xi(\lambda+\mu_{1}^{*})\Pi(0,0),\,x\in L^{+},\vspace{2mm}\\
\xi(\lambda+\mu_{1}^{*})\Pi(0,z_{2}(x))=&\frac{1}{2i\pi}\int_{\zeta\in L}J(\zeta)\frac{d\zeta}{\zeta-x}+\bar{\xi}(\lambda+\mu_{2}^{*})\Pi(0,0),\,x\in L^{-}.
\end{array}\label{1ert}
\end{equation}
Applying the Plemelj-Sokhotski formulas \cite{ga} we have for $x\in L$
\begin{equation}
\begin{array}{rl}
\bar{\xi}(\lambda+\mu_{2}^{*})\Pi(z_{1}(x),0)=&\frac{1}{2i\pi}\int_{\zeta\in L}J(\zeta)\frac{d\zeta}{\zeta-x}+\xi(\lambda+\mu_{1}^{*})\Pi(0,0)+\frac{J(x)}{2},\vspace{2mm}\\
\xi(\lambda+\mu_{1}^{*})\Pi(0,z_{2}(x))=&\frac{1}{2i\pi}\int_{\zeta\in L}J(\zeta)\frac{d\zeta}{\zeta-x}+\bar{\xi}(\lambda+\mu_{2}^{*})\Pi(0,0)-\frac{J(x)}{2}.
\end{array}
\label{g0}
\end{equation}
\section{The case of single arrivals}\label{sa}
The model with two arrival flows of jobs that enter the system, and brings only one customer at a time was investigated in detail in \cite{dimpeis} (see also \cite{dim2}). In \cite{dimpeis}, for the case of exponentially distributed service times, we investigated the stationary joint distribution of the number of jobs in orbits and in service station in terms of a solution of a Riemann-Hilbert boundary value problem. For the case of arbitrarily distributed times the pgf of the number of jobs in orbits was obtained with the aid of a Riemann boundary value problem. 

Thus, in the rest of this section, assume that Two independent Poisson streams of jobs, say $S_{1}$, $S_{2}$ enter the system. The arrival rate of stream $S_i$ is $\lambda_{i}$, $i = 1, 2$. In the following, we apply the power series approximation.
\subsection{Exponential service times}\label{exp1}
In such a case, following \cite{dimpeis}, by writing down the balance equations and come up with the following system of functional equations
\begin{equation}
\begin{array}{rl}
\alpha H^{(0)}(z_{1},z_{2})-\mu H^{(1)}(z_{1},z_{2})=&(\mu_{2}^{*}-\mu_{1}^{*})[\bar{\xi}H^{(0)}(z_{1},0)-\xi H^{(0)}(0,z_{2})]\\
&+[\bar{\xi}\mu_{1}^{*}+\xi\mu_{2}^{*}]H^{(0)}(0,0),
\end{array}
\label{e0111}
\end{equation}
\begin{equation}
\begin{array}{l}
(\lambda z_{1}z_{2}+\xi\mu_{1}^{*}z_{2}+\bar{\xi}\mu_{2}^{*}z_{1})H^{(0)}(z_{1},z_{2})-z_{1}z_{2}(\lambda-\lambda_{1}z_{1}-\lambda_{2}z_{2}+\mu)H^{(1)}(z_{1},z_{2})\\
=(\mu_{2}^{*}z_{1}-\mu_{1}^{*}z_{2})[\bar{\xi}H^{(0)}(z_{1},0)-\xi H^{(0)}(0,z_{2})]+[\bar{\xi}\mu_{1}^{*}z_{2}+\xi\mu_{2}^{*}z_{1}]H^{(0)}(0,0),
\end{array}
\label{e0211}
\end{equation} 
where $\alpha=\lambda+\xi\mu_{1}^{*}+\bar{\xi}\mu_{2}^{*}$, $\lambda=\lambda_{1}+\lambda_{2}$. After some algebra by using (\ref{e0111}), (\ref{e0211}) we obtain
\begin{equation}
K(z_{1},z_{2})H^{(0)}(z_{1},z_{2})=A(z_{1},z_{2})H^{(0)}(z_{1},0)+B(z_{1},z_{2})H^{(0)}(0,z_{2})+C(z_{1},z_{2})H^{(0)}(0,0),
\label{g01}
\end{equation}
where for $\widehat{\lambda}_{k}=\lambda_{k}\alpha$, $k=1,2,$
\begin{equation}
\begin{array}{rl}
K(z_{1},z_{2})=&(1-x)z_{2}[\widehat{\lambda}_{1}z_{1}-\mu\xi\mu_{1}^{*}]+(1-z_{2})z_{1}[\widehat{\lambda}_{2}z_{2}-\mu\bar{\xi}\mu_{2}^{*}],\\
A(z_{1},z_{2})=&\bar{\xi}\{(1-z_{1})z_{2}[\lambda_{1}(\mu_{2}^{*}-\mu_{1}^{*})z_{1}+\mu\mu_{1}^{*}]\\&+(1-z_{2})z_{1}[\lambda_{2}(\mu_{2}^{*}-\mu_{1}^{*})z_{2}-\mu\mu_{2}^{*}]\},\\
B(z_{1},z_{2})=&-\frac{\xi A(z_{1},z_{2})}{1-\xi},\\
C(z_{1},z_{2})=&(1-z_{1})z_{2}\{\lambda_{1}[\bar{\xi}\mu_{1}^{*}+\xi\mu_{2}^{*}]z_{1}-\bar{\xi}\mu\mu_{1}^{*}\}\\&+(1-z_{2})z_{1}\{\lambda_{2}[\bar{\xi}\mu_{1}^{*}+\xi\mu_{2}^{*}]z_{2}-\xi\mu\mu_{2}^{*}\}.
\end{array}
\label{ui10}
\end{equation}
For such a system we showed (see Theorem 1 in \cite{dimpeis}) that the system is stable when
\begin{equation}
\begin{array}{l}
\frac{\lambda}{\mu}[\frac{\lambda_{1}+\xi\mu_{1}^{*}}{\mu_{1}}+\frac{\lambda_{1}+\bar{\xi}\mu_{2}^{*}}{\mu_{2}}]<1\Leftrightarrow\rho:=\frac{\lambda_{1}}{\mu}(\frac{\lambda+\mu_{1}^{*}}{\mu_{1}})+\frac{\lambda_{2}}{\mu}(\frac{\lambda+\mu_{2}^{*}}{\mu_{2}})<1.
\label{sta0}
\end{array}
\end{equation}

Assume hereon that $\rho<1$. As we did in Section \ref{exp}, we will first construct a power series expansion of the pgf $H^{(0)}(z_{1},z_{2})$ in $\xi$, directly from (\ref{g01}) and then, using (\ref{e0111}) we construct the corresponding power series expansion for $H^{(1)}(z_{1},z_{2})$ in terms of those of $H^{(0)}(z_{1},z_{2})$. We will show that 
\begin{equation}
\begin{array}{c}
H^{(n)}(z_{1},z_{2};\xi):=H^{(0)}(z_{1},z_{2})=\sum_{m=0}^{\infty}V_{m}^{(n)}(z_{1},z_{2})\xi^{m},\,n=0,1,
\end{array}
\label{ps0}
\end{equation}
and we will show how we can obtain recursively the terms $V_{m}^{(n)}(z_{1},z_{2})$.
Equation (\ref{g01}) is rearranged as
\begin{equation}
\begin{array}{l}
G(z_{1},z_{2})H^{(0)}(z_{1},z_{2};\xi)-G_{10}(z_{1},z_{2})H^{(0)}(z_{1},0;\xi)-G_{0,0}(z_{1},z_{2})H^{(0)}(0,0;\xi)\\
=\xi G_{10}(z_{1},z_{2})[H^{(0)}(z_{1},z_{2};\xi)-H^{(0)}(z_{1},0;\xi)-H^{(0)}(0,z_{2};\xi)+H^{(0)}(0,0;\xi)],
\end{array}
\label{r01}
\end{equation}
where,
\begin{displaymath}
\begin{array}{rl}
G(z_{1},z_{2})=&z_{2}[\mu\mu_{2}^{*}+(\lambda+\mu_{2}^{*})(\lambda-\lambda_{1}z_{1}]-\mu\mu_{2}^{*}-\lambda_{2}(\lambda+\mu_{2}^{*})z_{2}^{2},\\
G_{10}(z_{1},z_{2})=&(\mu_{2}^{*}-\mu_{1}^{*})z_{2}[\lambda(1-z_{1})+\lambda(1-z_{2})]+\mu[\mu_{2}^{*}(z_{2}-1)-\mu_{1}^{*}(1-z_{1}^{-1})],\\
G_{00}(z_{1},z_{2})=&\mu_{1}^{*}z_{2}[(\lambda_{1}z_{1}-\mu)(z_{1}^{-1}-1)+\lambda_{2}(1-z_{2})].
\end{array}
\end{displaymath}
The following theorem summarizes our main result.
\begin{theorem}
Under stability condition (\ref{sta0}) the pgfs $H^{(n)}(z_{1},z_{2};\xi)$ can be written in power series expansions on $\xi$ with coefficients
\begin{equation}
\begin{array}{rl}
V_{0}^{(0)}(z_{1},z_{2})=&\frac{(1-\rho)[G_{00}(z_{1},z_{2})G_{10}(z_{1},Y_{0}(z_{1})))-G_{10}(z_{1},z_{2})G_{00}(z_{1},Y_{0}(z_{1})))]}{G(z_{1},z_{2})G_{10}(z_{1},Y_{0}(z_{1})))},\\V_{m}^{(0)}(z_{1},z_{2})=&\frac{G_{10}(z_{1},z_{2})F_{m-1}(z_{1},z_{2})}{G(z_{1},z_{2})},\,m\geq 1,
\end{array}\label{u00}
\end{equation}
\begin{equation}
\begin{array}{rl}
V_{0}^{(1)}(z_{1},z_{2})=&\frac{\lambda+\mu_{2}^{*}}{\mu}V_{0}^{(0)}(z_{1},z_{2})-\frac{\mu_{2}^{*}-\mu_{1}^{*}}{\mu}V_{0}^{(0)}(z_{1},0)-\frac{(1-\rho)\mu_{1}^{*}}{\mu},\\
V_{1}^{(1)}(z_{1},z_{2})=&\frac{\lambda+\mu_{2}^{*}}{\mu}V_{1}^{(0)}(z_{1},z_{2})-\frac{\mu_{2}^{*}-\mu_{1}^{*}}{\mu}[V_{0}^{(0)}(z_{1},z_{2})-V_{0}^{(0)}(z_{1},0)\\&-V_{0}^{(0)}(0,z_{2})+V_{1}^{(0)}(z_{1},0)+1-\rho],\\
V_{m}^{(1)}(z_{1},z_{2})=&\frac{\lambda+\mu_{2}^{*}}{\mu}V_{m}^{(0)}(z_{1},z_{2})\\&-\frac{\mu_{2}^{*}-\mu_{1}^{*}}{\mu}[V_{m-1}^{(0)}(z_{1},z_{2})-V_{m-1}^{(0)}(z_{1},0)+V_{m}^{(0)}(z_{1},0)],\,m\geq2,
\end{array}
\label{u01}
\end{equation}
where $Y_{0}(z_{1})$, $|z_{2}|\leq1$ is the only zero of $G(z_{1},z_{2})$ inside the unit disk $|z_{2}|\leq1$ and $F_{m}(z_{1},z_{2})=V_{m}^{(0)}(z_{1},z_{2})-V_{m}^{(0)}(z_{1},Y_{0}(z_{1}))-V_{m}^{(0)}(0,z_{2})+V_{m}^{(0)}(0,Y_{0}(z_{1}))$, $m\geq0$, with $F_{-1}(z_{1},z_{2})=0$.
\end{theorem}
\begin{proof}
The proof following the lines of Theorem \ref{tho}. Remind that $H^{(0)}(z_{1},z_{2};\xi)$ is analytic function of $\xi$ in a neighborhood of 0 (see \ref{appe}). We firstly express $H^{(0)}(z_{1},z_{2})$ by power series expansion in $\xi$ for all $x$, $y$ in the unit disk. Using (\ref{ps0}), (\ref{r01}) and equate the corresponding powers of $\xi$ at both sides yields
\begin{equation}
\begin{array}{rl}
G(z_{1},z_{2})V_{m}^{(0)}(z_{1},z_{2})=&G_{10}(z_{1},z_{2})[V_{m}^{(0)}(z_{1},0)+P_{m-1}(z_{1},z_{2})]\\&+G_{00}(z_{1},z_{2})V_{m}^{(0)}(0,0),\,m\geq0,
\end{array}\label{zm0}
\end{equation} 
where $P_{m}(z_{1},z_{2})=V_{m}^{(0)}(z_{1},z_{2})-V_{m}^{(0)}(z_{1},0)-V_{m}^{(0)}(0,z_{2})+V_{m}^{(0)}(0,0)$, $m\geq1$, and $P_{-1}(z_{1},z_{2})=0$. Using Rouch\'{e}'s theorem we can show that for $|z_{1}|\leq1$, $G(z_{1},z_{2})=0$ has a unique root, say $Y_{0}(z_{2})$ such that $|Y_{0}(z_{2})|<1$. Indeed, set $f(z_{2}):=z_{2}[\mu\mu_{2}^{*}+\lambda_{2}(\lambda+\mu_{2}^{*})+\lambda_{1}(\lambda+\mu_{2}^{*})(1-z_{1})]$, and $g(z_{2}):=-\mu\mu_{2}^{*}-\lambda_{2}(\lambda+\mu_{2}^{*})z_{2}^{2}$. Then
\begin{displaymath}
\begin{array}{l}
|f(z_{2})|=|\mu\mu_{2}^{*}+\lambda_{2}(\lambda+\mu_{2}^{*})+\lambda_{1}(\lambda+\mu_{2}^{*})(1-z_{1})|\\
\geq|\mu\mu_{2}^{*}+\lambda_{2}(\lambda+\mu_{2}^{*})|\geq|\mu\mu_{2}^{*}+\lambda_{2}(\lambda+\mu_{2}^{*})z_{2}^{2}|=|g(z_{2})|.
\end{array}
\end{displaymath}
A simple application of Rouch\'{e}'s theorem states that for $|z_{2}|=1$, $|Y_{0}(z_{1})|<1$. Moreover, for $z_{1}=1$, $Y_{0}(1)=min[1,\frac{\mu\mu_{2}^{*}}{\lambda_{2}(\lambda+\mu_{2}^{*})}]=1$, due to the stability conditions. Due to the implicit function theorem $Y_{0}(z_{1})$ is an analytic function in the unit disk, and $\frac{d}{dx}Y_{0}(z_{1})|_{z_{1}=1}=\frac{\lambda_{1}(\lambda+\mu_{2}^{*})}{\mu\mu_{2}^{*}-\lambda_{2}(\lambda+\mu_{2}^{*})}$. Due to the analyticity of $H^{(0)}(z_{1},z_{2})$ in the unit disk, the coefficients $V_{m}^{(0)}(z_{1},z_{2})$ are also analytic, and thus, the right hand side of (\ref{zm0}) vanishes for $z_{2}=Y_{0}(z_{1})$, and gives
\begin{equation}
\begin{array}{c}
V_{m}^{(0)}(z_{1},0)=-\frac{G_{00}(z_{1},Y_{0}(z_{1}))}{G_{10}(z_{1},Y_{0}(z_{1}))}V_{m}^{(0)}(0,0)-P_{m-1}(z_{1},Y_{0}(z_{1})).
\end{array}
\label{fgy}
\end{equation}
Using (\ref{fgy}), (\ref{zm0}) we obtain for $m\geq 0$,
\begin{equation}
\begin{array}{rl}
V_{m}^{(0)}(z_{1},z_{2})=&\frac{1}{G(z_{1},z_{2})}[\frac{G_{00}(z_{1},z_{2})G_{10}(z_{1},Y_{0}(z_{1}))-G_{10}(z_{1},z_{2})G_{00}(z_{1},Y_{0}(z_{1}))}{G_{10}(z_{1},Y_{0}(z_{1}))}V_{m}^{(0)}(0,0)\\&+G_{10}(z_{1},z_{2})F_{m-1}(z_{1},z_{2})].
\end{array}
\label{lp}
\end{equation}
We now need to obtain $V_{0}^{(0)}(0,0)$. This constant will be found by using the fact that $H^{(0)}(1,1;\xi)=1-\frac{\lambda}{\mu}$ (see \cite{dimpeis}, Lemma 1, pp. 144-145). This means that $V_{0}^{(0)}(1,1)=1-\frac{\lambda}{\mu}$, $V_{m}^{(0)}(1,1)=0$, $m\geq1$, and $Q_{m}(1,1)=0$, $m\geq0$. Using (\ref{lp}) for $m=0$, and setting $x=y=1$, we arrive after some algebra in $V_{0}^{(0)}(0,0)=1-\rho$, $V_{m}^{(0)}(0,0)=0$, $m\geq1$. With this part and using (\ref{lp}) we obtain (\ref{u00}). Now substitute (\ref{ps0}) in (\ref{e0111}) to obtain,
\begin{displaymath}
\begin{array}{l}
\sum_{m=0}^{\infty}V_{m}^{(1)}(z_{1},z_{2})\xi^{m}=\frac{\alpha}{\mu}\sum_{m=0}^{\infty}V_{m}^{(0)}(z_{1},z_{2})\xi^{m}-\frac{\mu_{2}^{*}-\mu_{1}^{*}}{\mu}\sum_{m=0}^{\infty}V_{m}^{(0)}(z_{1},0)\xi^{m}\\
+\frac{\mu_{2}^{*}-\mu_{1}^{*}}{\mu}\sum_{m=0}^{\infty}[V_{m}^{(0)}(z_{1},0)+V_{m}^{(0)}(0,z_{2})]\xi^{m+1}-\frac{(1-\rho)\mu_{1}^{*}}{\mu}-\frac{(\mu_{2}^{*}-\mu_{1}^{*})(1-\rho)\xi}{\mu}.
\end{array}
\end{displaymath}
Equate the coefficients of the corresponding powers in $\xi$ to obtain $V_{m}^{(1)}(z_{1},z_{2})$ in terms of $V_{m}^{(0)}(z_{1},z_{2})$, as given in (\ref{u00}).
\end{proof}

%Hence, starting from $V_{0}^{(0)}(z_{1},z_{2})$ in (\ref{u00}), we can iteratively determined all the functions $V_{m}^{(0)}(z_{1},z_{2})$, and subsequently using (\ref{u01}), the rest functions $V_{m}^{(1)}(z_{1},z_{2})$.
\subsection{The arbitrary distributed service times}
For the case of arbitrarily distributed service times we consider the embedded Markov chain at service completion epochs, and following \cite{dimpeis}, section 6, we have,
\begin{equation}
\begin{array}{rl}
(z_{1}z_{2}-\phi_{0}(z_{1},z_{2}))\Pi(z_{1},z_{2})=&\psi_{1}(z_{1},z_{2})\Pi(z_{1},0)+\psi_{2}(z_{1},z_{2})\Pi(0,z_{2})\\&+\psi_{3}(z_{1},z_{2})\Pi(0,0),
\end{array}
\label{fund0}
\end{equation}
where,
\begin{equation}
\begin{array}{rl}
\psi_{0}(z_{1},z_{2})=&\frac{[\xi\mu_{1}^{*}z_{2}+\mu_{2}^{*}(1-\xi)z_{1}+\lambda z_{1}z_{2}]\beta^{*}(\lambda-\lambda_{1}z_{1}-\lambda z_{2})}{\lambda+\mu_{2}^{*}-\xi(\mu_{2}^{*}-\mu_{1}^{*})},\\
\psi_{1}(z_{1},z_{2})=&(1-\xi)z_{1}\tilde{S}(z_{1},z_{2}),\\
\psi_{2}(z_{1},z_{2})=&-\frac{\xi}{1-\xi}\frac{\lambda+\mu_{1}^{*}}{\lambda+\mu_{2}^{*}}\psi_{1}(z_{1},z_{2}),\\
\psi_{3}(z_{1},z_{2})=&z_{1}(\tilde{T}_{0}(z_{1},z_{2})+\xi \tilde{T}_{1}(z_{1},z_{2})),\\
\tilde{T}_{0}(z_{1},z_{2})=&\frac{(\lambda+\mu_{2}^{*})\mu_{1}^{*}z_{2}(1-z_{1}^{-1})\beta^{*}(\lambda-\lambda_{1}z_{1}-\lambda_{2}z_{2})}{\lambda+\mu_{1}^{*}},\\
\tilde{S}(z_{1},z_{2})=&\frac{\lambda z_{2}(\mu_{2}^{*}-\mu_{1}^{*})-\lambda(\mu_{2}^{*}-\mu_{1}^{*}z_{2}z_{1}^{-1})+\mu_{1}^{*}\mu_{2}^{*}(z_{2}z_{1}^{-1}-1)}{\lambda+\mu_{1}^{*}}\beta^{*}(\lambda-\lambda_{1}z_{1}-\lambda_{2}z_{2}),\\
\tilde{T}_{1}(z_{1},z_{2})=&\frac{\mu_{1}^{*}z_{2}(1-z_{1}^{-1})(\mu_{2}^{*}-\mu_{1}^{*})\beta^{*}(\lambda-\lambda_{1}z_{1}-\lambda_{2}z_{2})}{\lambda+\mu_{1}^{*}}+\frac{\lambda+\mu_{1}^{*}}{\lambda+\mu_{2}^{*}}\tilde{S}(z_{1},z_{2}).
\end{array}
\end{equation}
Using results from \cite{dimpeis}, our system is stable when
\begin{equation}
\rho=\bar{b}(\frac{\lambda_{1}(\lambda+\mu_{1}^{*})}{\mu_{1}^{*}}+\frac{\lambda_{2}(\lambda+\mu_{2}^{*})}{\mu_{2}^{*}})<1.
\label{stabb}
\end{equation}
Following the procedure developed in Section \ref{gen}, we apply the power series approximation. The following theorem summarizes the main result.
\begin{theorem}\label{the}
Under stability condition (\ref{stabb}), the pgf $\Pi(z_{1},z_{2},\xi)$ can be written in power series expansions on $\xi$ with coefficients
\begin{equation}
\begin{array}{rl}
V_{0}(z_{1},z_{2})=&\frac{(1-\rho)[\tilde{T}_{0}(z_{1},z_{2})\tilde{S}(z_{1},Y_{0}(z_{1}))-\tilde{S}(z_{1},z_{2})\tilde{T}_{0}(z_{1},Y_{0}(z_{1}))]}{\tilde{U}(z_{1},z_{2})\tilde{S}(z_{1},Y_{0}(z_{1}))},\\
V_{m}(z_{1},z_{2})=&\frac{\tilde{F}_{m-1}(z_{1},z_{2})}{\tilde{U}(z_{1},z_{2})\tilde{S}(z_{1},Y_{0}(z_{1}))},\,m\geq 1,
\end{array}
\label{u0as}
\end{equation}
where for $m\geq0$,
\begin{displaymath}
\begin{array}{rl}
\tilde{F}_{m}(z_{1},z_{2})=&\tilde{U}_{0}(z_{1},z_{2})\tilde{S}(z_{1},Y_{0}(z_{1}))V_{m}(z_{1},z_{2})\\&-\tilde{U}_{0}(z_{1},Y_{0}(z_{1}))\tilde{S}(z_{1},z_{2})V_{m}(z_{1},Y_{0}(z_{1}))\\&+\frac{(\lambda+\mu_{1}^{*})\tilde{S}(z_{1},z_{2})\tilde{S}(z_{1},Y_{0}(z_{1}))}{\lambda+\mu_{1}^{*}}[V_{m}(0,z_{2})-V_{m}(0,Y_{0}(z_{1}))]\\&
+V_{m}(0,0)\frac{\mu_{1}^{*}(\mu_{2}^{*}-\mu_{1}^{*})(z_{1}^{-1}-1)}{\lambda+\mu_{1}^{*}}[Y_{0}(z_{1})\tilde{S}(z_{1},Y_{0}(z_{1}))-z_{2}\tilde{S}(z_{1},Y_{0}(z_{1}))],
\end{array}
\end{displaymath}
with $Q_{-1}(z_{1},z_{2})=0$ and,
\begin{displaymath}
\begin{array}{rl}
\tilde{U}(z_{1},z_{2})=&z_{2}(\lambda+\mu_{2}^{*})-(\mu_{2}^{*}+\lambda z_{2})\beta^{*}(\lambda-\lambda_{1}z_{1}-\lambda_{2}z_{2}),\\
\tilde{U}_{0}(z_{1},z_{2})=&\mu_{2}^{*}(z_{2}-\beta^{*}(\lambda-\lambda_{1}z_{1}-\lambda_{2}z_{2}))-\mu_{1}^{*}z_{2}(1-\frac{\beta^{*}(\lambda-\lambda_{1}z_{1}-\lambda_{2}z_{2})}{z_{1}}),
\end{array}
\end{displaymath}
and $Y_{0}(z_{1})$, is the only root of $\tilde{U}(z_{1},z_{2})=0$ for $|z_{1}|=1$, $|z_{2}|\leq1$.
\end{theorem}
\begin{proof}
The proof is similar to the one of Theorem \ref{th1} and further details are omitted.
\end{proof}
\subsection{Special cases}
Let $\mu_{k}^{*}=\mu^{*}$, and $z_{1}=z_{2}$, i.e., the ordinary single class M/M/1 retrial system under constant retrial policy. Using equation (2.6) from \cite{dimpeis} we obtain
\begin{displaymath}
\begin{array}{l}
H^{(0)}(z_{1},z_{1})=\frac{\mu^{*}(\mu-\lambda z_{1})(1-\frac{\lambda(\lambda+\mu^{*})}{\mu\mu^{*}})}{\mu\mu^{*}-\lambda(\lambda+\mu^{*})},\,H^{(1)}(z_{1},z_{1})=\frac{(\lambda+\mu^{*})H^{(0)}(z_{1},z_{1})-\mu^{*}(1-\frac{\lambda(\lambda+\mu^{*})}{\mu\mu^{*}})}{\mu}.
\end{array}
\end{displaymath}
On the other hand, using our PSA approach, the equations (\ref{u0}), (\ref{u1}) yields $V_{0}^{(n)}(z_{1},z_{2})=H^{(n)}(z_{1},z_{2})$, $V_{m}^{(n)}(z_{1},z_{2})=0$, $m\geq1$, $n=0,1$. This is expected since in such a system the jobs repeat their attempt according to a constant retrial policy with rate $\mu^{*}$ irrespective of $\xi$.

Note also that for $\xi=0$, the orbit queue 2 becomes an ordinary queue in front of the server, which in turns mean that we give absolute \textit{priority} to that queue. In this case, $H^{(0)}(z_{1},z_{2};0)=V_{0}^{(0)}(z_{1},z_{2})$ as given in (\ref{u0}). Indeed, setting $\xi=0$ to (\ref{g1}), and realizing that $B(z_{1},z_{2})=0$ we obtain,
\begin{equation}
\begin{array}{c}
\tilde{K}(z_{1},z_{2})H^{(0)}(z_{1},z_{2})=\tilde{A}(z_{1},z_{2})H^{(0)}(z_{1},0)+\tilde{C}(z_{1},z_{2})(1-\rho),
\end{array}
\label{g11}
\end{equation}
where now $\tilde{K}(z_{1},z_{2})=z_{1}G(z_{1},z_{2})$, $\tilde{A}(z_{1},z_{2})=z_{1}G_{10}()z_{1},z_{2}$, $\tilde{C}(z_{1},z_{2})=z_{1}G_{00}(z_{1},z_{2})$. Taking into account the only root $y=Y_{0}(x)$ of $\tilde{K}(z_{1},z_{2})$ in the unit disk, it is straightforward that $H^{(0)}(z_{1},z_{2})=V_{0}^{(0)}(z_{1},z_{2})$ as given in (\ref{u00}).
\section{Performance metrics}\label{perf}
%We have obtained an algorithm to determine the $V_{m}^{(n)}(x,y)$ for each $m$, $n=0,1$. 
We will discuss only the case of exponentially distributed service times (Section \ref{exp}), since we believe that it is more interesting, due to the fact that for suc a case the system is described by a 3-dimensional Markov process, i.e., a modulated random walk in the quarter plane. Similar discussion can be done for the case of arbitrarily distributed service times. Clearly,
\begin{equation}
\begin{array}{rl}
E(N_{1})=&\sum_{m=0}^{\infty}\frac{\partial}{\partial z_{1}}[V_{m}^{(0)}(z_{1},z_{2})+V_{m}^{(1)}(z_{1},z_{2})]\xi^{m}|_{z_{1}=z_{2}=1},\\E(N_{2})=&\sum_{m=0}^{\infty}\frac{\partial}{\partial z_{2}}[V_{m}^{(0)}(z_{1},z_{2})+V_{m}^{(1)}(z_{1},z_{2})]\xi^{m}|_{z_{1}=z_{2}=1}.
\end{array}
\label{pss}
\end{equation}
Note that from (\ref{u1}),
\begin{displaymath}
\begin{array}{rl}
\frac{\partial}{\partial z_{1}}V_{m}^{(1)}(z_{1},z_{2})=&\frac{\alpha}{\mu}\frac{\partial}{\partial z_{1}}V_{m}^{(0)}(z_{1},z_{2})+\frac{\mu_{2}^{*}-\mu_{1}^{*}}{\mu}[\frac{\partial}{\partial z_{1}}V_{m-1}^{(0)}(z_{1},0)\mathbf{1}_{\{m\geq1\}}-\frac{\partial}{\partial z_{1}}V_{m}^{(0)}(z_{1},0)],\\
\frac{\partial}{\partial z_{2}}V_{m}^{(1)}(z_{1},z_{2})=&\frac{\alpha}{\mu}\frac{\partial}{\partial z_{2}}V_{m}^{(0)}(z_{1},z_{2})+\frac{\mu_{2}^{*}-\mu_{1}^{*}}{\mu}\frac{\partial}{\partial z_{2}}V_{m-1}^{(0)}(0,z_{2})\mathbf{1}_{\{m\geq1\}},
\end{array}
\end{displaymath}
where $\mathbf{1}_{X}$ the indicator function of the event $X$. Truncation of the power series (\ref{pss}) yields
\begin{equation}
\begin{array}{c}
E(N_{k})=\sum_{m=0}^{M}v_{k,m}\xi^{m}+O(\xi^{M+1}),\,k=1,2,
\end{array}
\label{tru}
\end{equation}
for $v_{1,m}=\frac{\partial}{\partial z_{1}}[V_{m}^{(0)}(z_{1},z_{2})+V_{m}^{(1)}(z_{1},z_{2})]|_{z_{1}=z_{2}=1}$, $v_{2,m}=\frac{\partial}{\partial z_{2}}[V_{m}^{(0)}(z_{1},z_{2})+V_{m}^{(1)}(z_{1},z_{2})]|_{z_{1}=z_{2}=1}$, $m=0,1,...,M$. Numerical results shown that such a truncation yields accurate approximations for small values of $\xi$ (see Figure \ref{dw} (left))\footnote{The problem is symmetric in $\xi$ by constructing PSA in $1-\xi$, instead of $\xi$.}. Alternatively, one can use Pad\'{e} approximants, which replace (\ref{pss}) by a rational functional. In such a case $E(N_{k})$ are replaced by 
$[L/N]_{E(N_{k})}(\xi)=\frac{\sum_{l=0}^{L}v_{k,l}\xi^{l}}{\sum_{n=0}^{N}w_{k,n}\xi^{n}}$
with $N$, $L$, $v_{k,l}$,$w_{k,n}$ appropriately chosen; see \cite{walr} for more details. Numerical results have shown that Pad\'{e} approximants have been proven very accurate for any $\xi\in[0,1]$ (see Figure \ref{dw} (right)).

\section{Numerical illustration}\label{num}
\subsection{The case of single arrivals-exponentially distributed service times}
We now compare the PSA approximations (subsection \ref{exp1}) with the exact results obtained in \cite{dimpeis} in order to investigate the influence of some parameters on the mean queue lengths. We focus only on $E(N_{2})$; similar results stands for $E(N_{1})$.

Figure \ref{dw} (left) ($\lambda_{1}=1$, $\lambda_{2}=2.2$, $\mu_{1}^{*}=8$, $\mu_{2}^{*}=10$, $\mu=5$) depicts the approximations (\ref{tru}) as a function of $\xi$ for varying values of $M$. It can be easily seen that for $\xi$ close to zero, PSA is very accurate, and by increasing the number of terms the accuracy becomes even better. It can be also observed that by increasing the number of terms, we can have larger regions of $\xi$ with good accuracy of the approximations. The above observations are validated by comparing the PSA with the exact results derived in \cite{dimpeis}, Section 4. However, we can also see that truncation approximations deteriorate for $\xi$ away from 0 (or by symmetry, from 1). For such a case, Pad\'{e} approximants shown extremely good performance. Indeed, in Figure \ref{dw} (right) we can observe the high level of accuracy of Pad\'{e} approximants for the same example, and for $N=2$, $L=2M+1-N=7$, with $M=4$. 
\begin{figure}[ht!]
\centering
\begin{minipage}[b]{0.45\linewidth}
\includegraphics[scale=0.46]{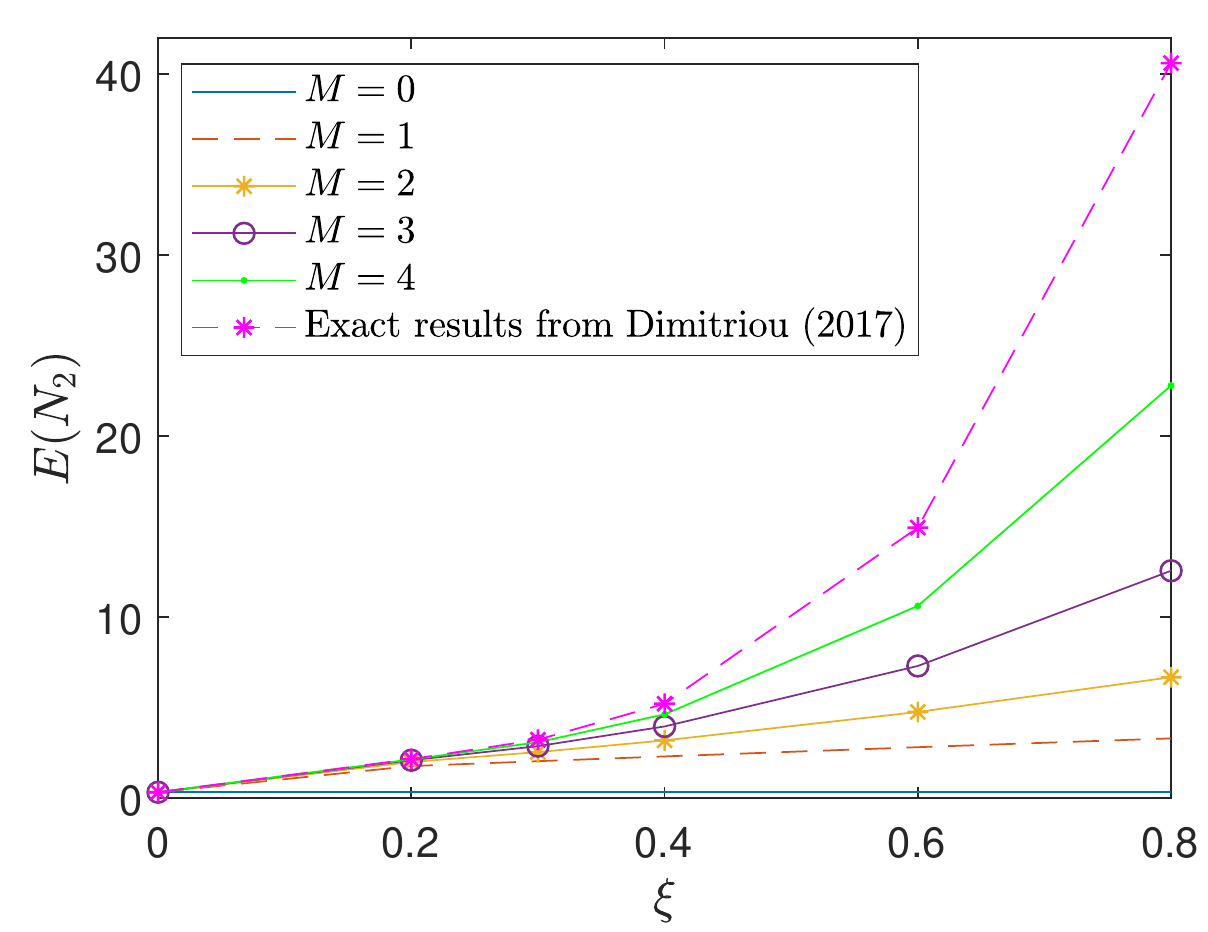}
\end{minipage}
\quad
\begin{minipage}[b]{0.5\linewidth}
\includegraphics[scale=0.54]{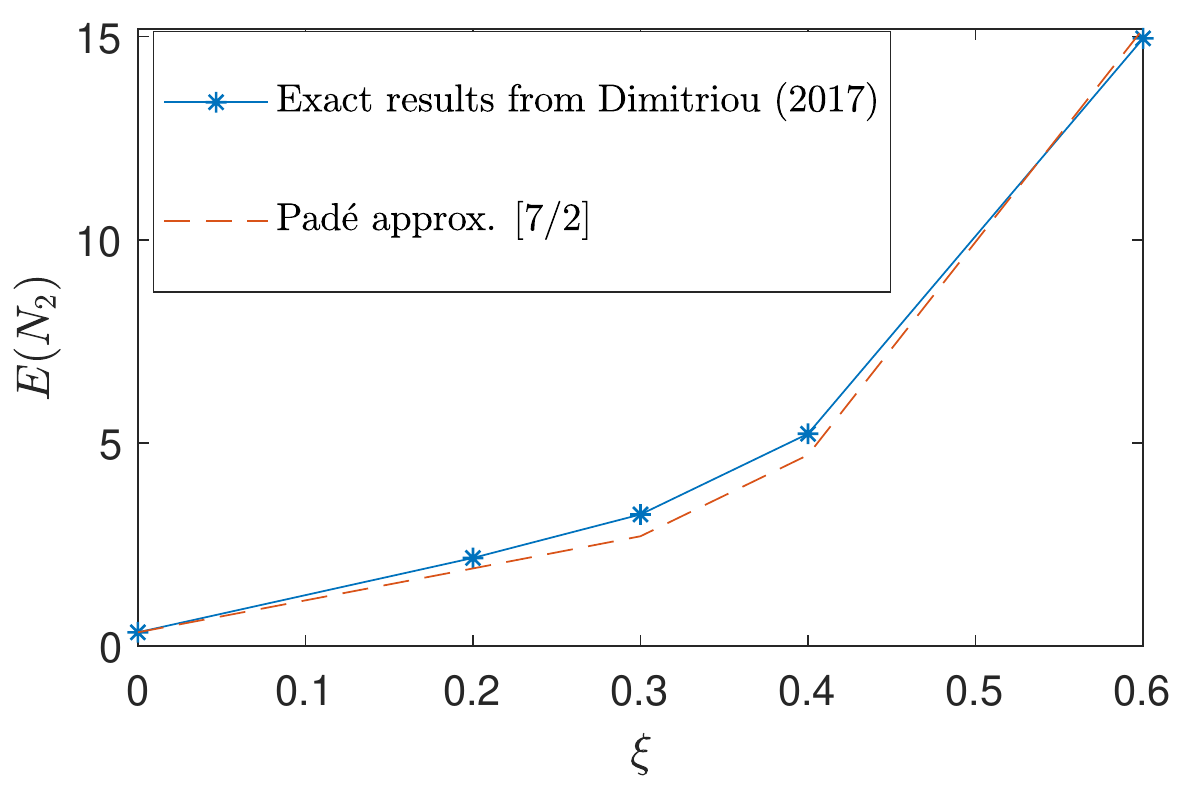}
\end{minipage}\caption{Truncation (left) and Pad\'{e} approx. of $E(N_{2})$ for varying $\xi$.}
\label{dw}
\end{figure}

Finally, Figure \ref{fig1} shows how $E(N_{2})$ increases for increasing values of $\lambda_{2}$ and varying values of $\mu$, for $M=1000$ and $\xi=0.1$ ($\lambda_{1}=1$, $\mu_{1}^{*}=8$, $\mu_{2}^{*}=10$). We can see that for such a value of $\xi$ the increase in $E(N_{2})$ is quite smooth, although the system is close to saturation for $\lambda_{2}$ near to 2.2 ($\rho=0.87$). This is because for such a $\xi$, the system is near to a priority for the orbit queue 2.
\begin{figure}[ht]
\centering
\includegraphics[scale=0.5]{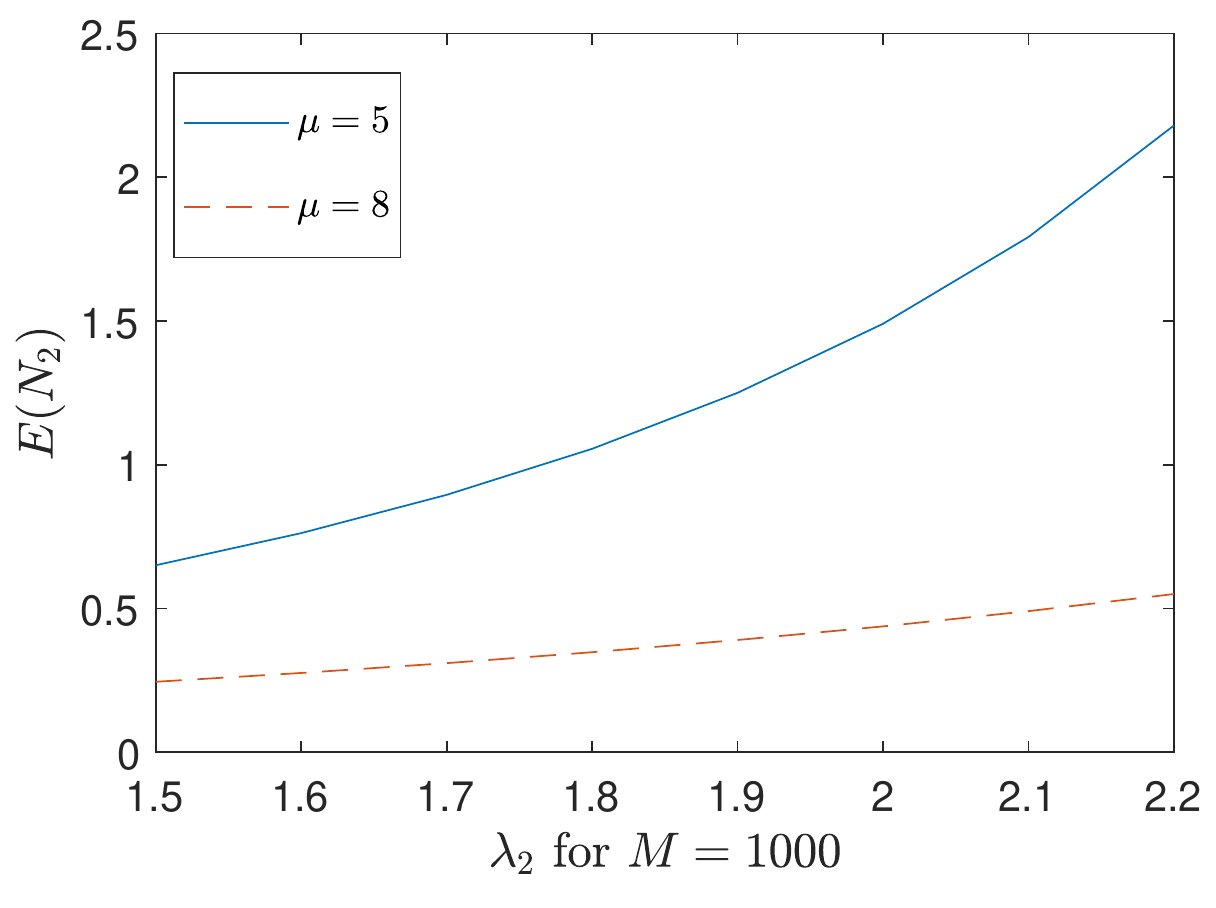}
\caption{Truncation approx. of $E(N_{2})$ for varying $\lambda_{2}$ with $\xi=0.1$.}
\label{fig1}
\end{figure}
%\begin{figure}[tp]
%\centering
%\includegraphics[scale=0.8]{fig2c.pdf}
%\epsfxsize 15cm \epsfysize 20cm \centerline{\epsfbox{doc1.eps}}
%\caption{Probability of an empty system for $\mu_{1}=2$.}
%\label{fig1}
%\end{figure}
\subsection{The case of structured batch arrivals-exponentially distributed service times}
In the following we consider the case of structured batch arrivals (Section \ref{exp}), and assume that if $W$ denotes the batch size and $X_{mi}$, $i = 1,2$, the number of $P_i$ customers in a batch of size $m$, then for $m=1,2,...,$
\begin{displaymath}
\begin{array}{c}
Pr(W=m)=\frac{1}{2^{m}},\,Pr(X_{m1}=k_{1},X_{m2}=k_{2})=\frac{m!}{k_{1}!k_{2}!}u_{1}^{k_{1}}u_{2}^{k_{2}},
\end{array}
\end{displaymath}
with $k_{1}+k_{2}=m$, $u_{1}+u_{2}=1$. Thus,
\begin{displaymath}
\begin{array}{c}
G(z_{1},z_{2})=\sum_{m=1}^{\infty}\frac{1}{2^{m}}(u_{1}z_{1}+u_{2}z_{2})^{m}.
\end{array}
\end{displaymath}
Set $u_{1}=0.6$, $\mu_{1}^{*}=8$, $\mu_{2}^{*}=10$, $\mu=5$, $\xi=0.1$, $p_{1}=0.6$.
\begin{figure}[ht!]
\centering
\includegraphics[scale=0.5]{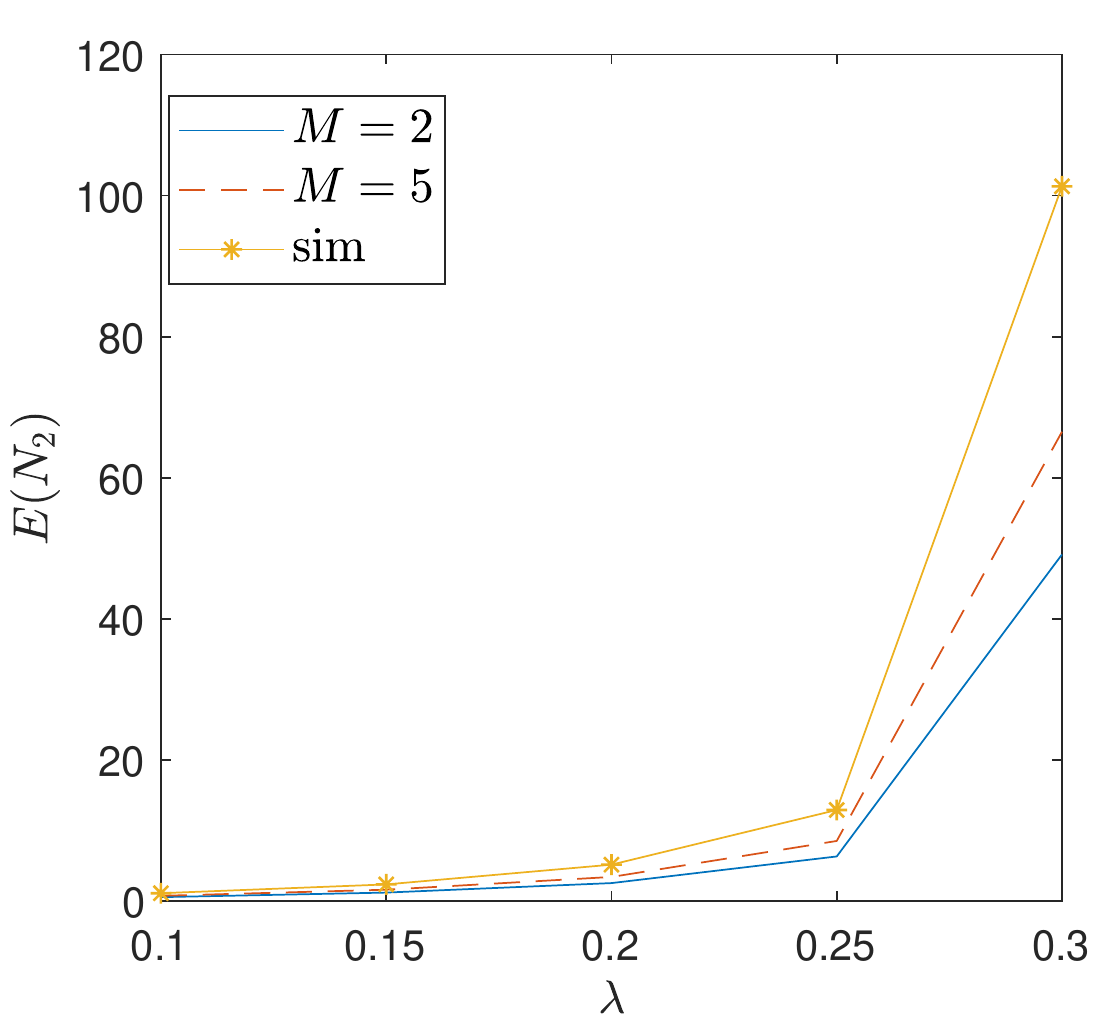}
\caption{Truncation approx. of $E(N_{2})$ for varying $\lambda$ with $\xi=0.1$.}
\label{figa1}
\end{figure}

In Figure \ref{figa1} we compare the PSA approximations to simulation results and investigate the influence of some parameters on the mean queue lengths. We easily observe that for $\xi$ close to zero, PSA is very accurate, and by increasing the number of terms the accuracy becomes even better.

\section{Conclusion}
In this paper we provided power series expansions of the pgfs of the joint stationary orbit queue length distribution in two-class retrial systems with coupled orbit queues. In such a system, an orbit queue re-configures its re-transmission parameter as a function of a parameter $\xi$ ($0\leq\xi\leq 1$), based on the the state of the other orbit. We investigated the structured batch arrival system for both arbitrarily and exponentially distributed system. In such a system jobs arrive in batches of random size, which may include both types of jobs. We also provided the analysis for the case of two independent Poisson streams and single arrivals. Similar analysis can be applied for the case of independent batch arrivals. We distinguished the analysis between arbitrarily and exponentially distributed service times since in the latter case instead of a single functional equation, we have a system of functional equations. Thus, in the later case we obtain power series expansions of the pgfs $H(x,y)=(H^{(0)}(x,y),H^{(1)}(x,y))$ by first obtaining the coefficients of the power terms of $H^{(0)}(x,y)$, iteratively from a constant term, which corresponds to the case $\xi=0$ (which in turn refers to the priority retrial system), and then, obtain those of the power terms of $H^{(1)}(x,y)$, as functions of the coefficients of $H^{(0)}(x,y)$. In the former case, we provided power series expansion of the pgf $\Pi(z_{1},z_{2})$. 

Moreover, we also provided an alternative derivation of $\Pi(z_{1},z_{2})$ by formulating and solving a Riemann boundary value problem. Under such a setting, we can easily observe the advantages of the PSA approach compared with the boundary value approach, at least in terms of computational point of view.

Comparing our results with the exact derivations in \cite{dimpeis}, and with simulations we find accurate approximations by truncating the power series. The major advantage of our approach is that the basic performance metrics are obtained explicitly by the input parameters without additional computational effort.

In a future work, we plan to apply this approach to even general systems to include server failures, and more complex boundary behavior. An additional interesting topic is to apply this approach to a system with more than two orbit queues, in which the theory of boundary value problems cannot be applied. 
\appendix
\section{On the analyticity of $H^{(n)}(z_{1},z_{2};\xi)$ near to $\xi=0$}\label{appe}
We only need to focus on the analyticity of $H^{(0)}(z_{1},z_{2};\xi)$ in a neighborhood of $\xi=0$ by using a variant of the implicit function theorem on the functional equation (\ref{g1}). Then, the analyticity of $H^{(1)}(z_{1},z_{2};\xi)$, follows by the analyticity of $H^{(0)}(z_{1},z_{2};\xi)$ by using (\ref{e111}). For such an approach we use the implicit function theorem for Banach spaces (see Theorem 10.2.3, p. 272 in \cite{ban}, see also \cite{walr}). Define the mapping $f:S\subset\mathbb{C}\times B_{2}\to B_{3}\times\mathbb{C}$,
\begin{displaymath}
\begin{array}{c}
f(\xi,H^{(0)})=[K(z_{1},z_{2},\xi)H^{(0)}(x,y)-K_{1}(z_{1},z_{2},\xi)H^{(0)}(z_{1},z_{2})-K_{2}(z_{1},z_{2},\xi)H^{(0)}(0,z_{2})\\-K_{3}(z_{1},z_{2},\xi)H^{(0)}(0,0), H^{(0)}(1,1)-\frac{\mu-\lambda(2-G(0,1)-G(1,0))}{\mu-\lambda(1-G(0,1)-G(1,0))}],
\end{array}
\end{displaymath}
where $S$ contains the point $(0,V_{0}^{(0)})$, $K$, $K_{1}$, $K_{2}$, $K_{3}$, are as in (\ref{ui1}), $B_2$ be the Banach space comprising all bivariate analytic bounded
functions in $\mathbb{D}^{2}$, with $\mathbb{D}$ the open complex unit disk, and $B_3$ be the Banach
space comprising all trivariate analytic bounded functions in $\mathbb{D}^{3}$ that have a
limit of 0 for the first two arguments going to 1.

Since $K$, $A$, $B$, $C$ are bounded analytic functions in $\mathbb{D}^{3}$, and since $f$ is affine in $H^{(0)}$ and $\xi$, it is easily seen that $f$ is $r$-times continuously differentiable for all $r$. Note also that $f (0,V_{0}^{(0)}) = [0, 0]$.
Then, the (Banach space) derivative of $f$ at the point $(0,V_{0}^{(0)})$ \cite{ban} equals
\begin{displaymath}
\begin{array}{r}
df(0,V_{0}^{(0)})=[K(z_{1},z_{2},0)H^{(0)}(z_{1},z_{2})-K_{1}(z_{1},z_{2},0)H^{(0)}(z_{1},0)\\-K_{3}(z_{1},z_{2},0)H^{(0)}(0,0), H^{(0)}(1,1)].
\end{array}
\end{displaymath}
The next step is to show that this mapping is a homeomorphism. Indeed,
\begin{enumerate}
\item $df (0,V_{0}^{(0)})$ is a continuous mapping for the same reasons
that the mapping $f$ itself is continuous.
\item For given $H^{(0)}_{1}$, $H^{(0)}_{2}$, let $df (0,V_{0}^{(0)})(H^{(0)}_{1})=df (0,V_{0}^{(0)})(H^{(0)}_{2})$. Then,
\begin{displaymath}
\begin{array}{r}
K(z_{1},z_{2},0)(H^{(0)}_{1}(z_{1},z_{2})-H^{(0)}_{2}(z_{1},z_{2}))-K_{1}(z_{1},z_{2},0)(H^{(0)}_{1}(z_{1},0)-H^{(0)}_{2}(z_{1},0))\\-K_{3}(z_{1},z_{2},0)(H^{(0)}_{1}(0,0)-H^{(0)}_{2}(0,0))=0,\\
H^{(0)}_{1}(1,1)-H^{(0)}_{2}(1,1)=0.
\end{array}
\end{displaymath}
or equivalently $f(0,H^{(0)}_{1}-H^{(0)}_{2})=(0,\rho-1)$, which in turn has the zero solution as a unique solution \cite{asm}, and thus $H^{(0)}_{1}=H^{(0)}_{2}$ so that $df (0,V_{0}^{(0)})$ is injective.
\item To show that $df (0,V_{0}^{(0)})$ is surjective, we solve the $df (0,V_{0}^{(0)})(H^{(0)}) = (g,c)$ with $g$ a bivariate analytic bounded function in $\mathbb{D}^{2}$ with limit 0 for its arguments going to 1, and $c$ a complex number. The solution is
\begin{displaymath}
\begin{array}{l}
H^{(0)}(x,y)\\=\frac{g(x,y)K_{1}(z_{1},Y_{0}(z_{1}))-g(x,Y_{0}(x))K_{1}(z_{1},z_{2})+H^{(0)}(0,0)[K_{3}(z_{1},z_{2})K_{1}(z_{1},Y_{0}(z_{1}))-K_{3}(z_{1},Y_{0}(z_{1}))K_{1}(z_{1},z_{2})]}{K(z_{1},z_{2})K_{1}(z_{1},Y_{0}(z_{1}))}.
\end{array}
\end{displaymath}
\item The $H^{(0)}$ obtained previously equals $(df (0,V_{0}^{(0)}))^{-1}$, which is readily seen that it is continuous.
\end{enumerate}
Thus, $H^{(0)}\to df(0,V_{0}^{(0)})(H^{(0)})$ is a linear homeomorphism and using Theorem 10.2.3 in \cite{ban}, $H^{(0)}(z_{1},z_{2};\xi)$ is $r$-times differentiable at $\xi=0$. Having this result, and using (\ref{e111}), $H^{(1)}(z_{1},z_{2};\xi)$ is also $r$-times differentiable at $\xi=0$.

%Our approach go as follows: the left-hand side of the functional equation
%(\ref{g1}) is a mapping f (β,U(z1,z2)) from the product space of two Banach spaces
%to another Banach space. If this mapping satisfies certain conditions (most notably
%that the functional derivative D2(f )(β,U(z1,z2)) is a linear homeomorphism) and if
%there is a U0(z1,z2) so that f (0,U0(z1,z2)) = 0 then there exists a unique (analytic)
%function W(z1,z2,β) so that f (β,W(z1,z2,β)) = 0. In our case, we have a solution
%U0(z1,z2) = V0(z1,z2), see expression (14).
 
%\section*{References}

%\bibliographystyle{plain}

%\bibliography{mybibfile}
%

%\bibliography{mybibfile}

\end{document}